\documentclass[11pt,reqno]{amsart}
\usepackage[utf8]{inputenc}
\usepackage[T1]{fontenc}
\usepackage{lmodern}
\usepackage[english]{babel}
\usepackage[margin=2.55cm]{geometry}
\usepackage[T1]{fontenc}
\usepackage{lmodern,microtype,amsmath,amssymb,amsthm,mathtools,stmaryrd}
\usepackage{enumitem,booktabs,tabularx,hyperref}
\hypersetup{colorlinks=true}

\numberwithin{equation}{section}

 \usepackage[maxbibnames=99,backend=biber, sorting=nyt, url=false, giveninits=true, isbn=false, style=alphabetic]{biblatex}
\newtheorem{theorem}{Theorem}[section]
\newtheorem*{theorem*}{Theorem}
\newtheorem{proposition}[theorem]{Proposition}
\newtheorem{lemma}[theorem]{Lemma}
\newtheorem{corollary}[theorem]{Corollary}
\theoremstyle{remark}\newtheorem{remark}[theorem]{Remark}
\newtheorem{definition}[theorem]{Definition}
\newcommand{\N}{\mathbb N}
\newcommand{\R}{\mathbb R}
\newcommand{\HH}{\mathcal H}

\newcommand{\diver}{\operatorname{div}}
\newcommand{\reach}{\operatorname{reach}}
\newcommand{\spt}{\operatorname{spt}}
\newcommand{\Lip}{\operatorname{Lip}}

\newcommand{\cJ}{\mathcal J}

\title{The Brakke-BV structure of flat flows}

\author[D. De Gennaro]{Daniele De Gennaro}
\address{Dipartimento di Matematica, Universit\`{a} degli Studi di Milano, Via Saldini 50, I-20133 Milano (MI), Italy}
\email{daniele.degennaro@unimi.it}

\author[S. Stuvard]{Salvatore Stuvard}
\address{Dipartimento di Matematica, Universit\`{a} degli Studi di Milano, Via Saldini 50, I-20133 Milano (MI), Italy}
\email{salvatore.stuvard@unimi.it}

\date{\today}

\begin{document}

\begin{abstract}
    We show that, in ambient dimension at most seven, varifold limits of the minimizing movements scheme introduced by Almgren–Taylor–Wang and Luckhaus–Sturzenhecker are integral Brakke flows. As a consequence, we show that $L^1$-limit points of the scheme, usually known as flat flows, are generalized BV flows. The result is unconditional and, in particular, remains valid even in the absence of energy convergence and in the presence of fattening.
\end{abstract}

\maketitle\vspace{-0.9cm}

\tableofcontents

\section{Introduction}
\subsection{Minimizing movements and flat flows}
The mean curvature flow is a geometric evolution flow which arises as the  gradient flow of the area functional. It is therefore one of the most fundamental geometric flows involving extrinsic
curvatures. 
A smooth family of sets $\{E_t\}_{t\geq 0}$ with $E_t\subset \mathbb{R}^{n+1}$ evolves
by mean curvature if the normal velocity $V$ of $\partial E_t$ satisfies
\begin{equation}\label{intro:MCF}
    V(x,t)=\mathbf H_{E_t}(x),\quad \text{for all }x\in\partial E_t\,,
\end{equation}
where $\mathbf H_E=-H_E\nu_E$ is the mean curvature vector of a smooth set $E$.
This evolution has the formal structure of the
$L^2$-gradient flow of the perimeter functional $P(E)$: as long as the solution remains smooth, it satisfies
\[
    \frac{d}{dt} P(E_t)
    =
    -\int_{\partial E_t} |\mathbf{H}_{E_t}|^2
    \,d\mathcal{H}^n\,.
\]
It is well-known that smooth mean curvature flows typically develop
singularities in finite time. This has led to several notions of weak
solution, including Brakke flows \cite{Bra}, level-set solutions
\cite{EvaSpr,CheGigGot,ChaMorPonCPAM17}, and variational or distributional
formulations  \cite{LS,MSS}. A basic problem in the
subject is to understand to what extent these different weak formulations solve the smooth evolution equation \eqref{intro:MCF}  and how they are
related after singularities occur.

A natural way of preserving the variational structure beyond singularities
is the minimizing-movements method \cite{DeG,AmbGigSav}. The main issue in the present setting  is that the gradient flow identification is only formal, and one needs to introduce a suitable proxy for the $L^2$-distance between hypersurfaces when defining the discretization scheme. 
For mean curvature flow this approach was introduced independently by
Almgren--Taylor--Wang and Luckhaus--Sturzenhecker \cite{ATW,LS}. 
Given a time step $h>0$ and a previous set $F$, one iteratively minimizes
\begin{equation}\label{intro:scheme}
    E\mapsto P(E)+\frac{1}{h}D(E,F)\,,
    \qquad
    D(E,F)
    :=
    \int_{E\triangle F}\operatorname{dist}(x,\partial F)\,dx\,.
\end{equation} 
The quantity $D(E,F)$ is a geometric proxy for the  squared distance appearing
in the implicit Euler scheme for a gradient flow. Indeed, if $\partial E$
is obtained from $\partial F$ by a normal displacement $hV$, then formally
\[
    \frac{1}{h}D(E,F)
    \simeq
    \frac{h}{2}
    \int_{\partial F}|V|^2\,d\mathcal{H}^n\,.
\]
Given an initial set $E_0$, and upon fixing the time step $h > 0$, one iteratively defines a sequence of sets $\{E^h_k\}_{k\ge 0}$ so that $E^h_0=E_0$ and $E^h_{k+1}$ minimizes \eqref{intro:scheme} with $F$ replaced by $E^h_k$ for $k\ge 0$.  The approximate evolution, referred to as the \textit{discrete flow}, is then the family of sets $E^h_t:=E^h_{[t/h]}$ for $t\ge 0$, where $[\cdot]$ denotes the integer part. Any cluster point in $L^1$ of $\{E^h_t\}_{t\ge 0}$ is then called a \textit{flat flow} solution to \eqref{intro:MCF}.
It has been known since \cite{ATW,LS} that the main difficulty is not to construct such limit points of the
scheme, but to identify the equation they satisfy, if any.

Indeed, standard estimates give compactness of the piecewise-constant
interpolations in $L^1$, together with a uniform $1/2$-H\"older estimate in
time.
The delicate point is that one needs to work with $L^1$ convergence of the sets, whereas the Euler--Lagrange equation and
the velocity law involve
curvature terms and cannot be passed to the limit to determine the evolution speed.  
Luckhaus--Sturzenhecker \cite{LS} proved that, under the additional assumption that
the time-integrated perimeters converge
\begin{equation}\label{intro:ene-conv}
    \int_0^T P(E^h_t)\,dt\to \int_0^T P(E_t)\,dt, \quad \text{as }h\to 0\,,
\end{equation}
then flat flows satisfy a
distributional formulation of mean curvature flow. Their notion of solution is usually referred to as  \textit{BV solution} to  \eqref{intro:MCF}.
Note that   assumption \eqref{intro:ene-conv} is nontrivial to verify, as $L^1$-convergence of the sets only implies lower semicontinuity of the perimeters as  the cancellation mechanism of currents can happen.

Variants of this argument have subsequently been developed for the
volume-preserving flow  \cite{MSS,Jul24}, for anisotropic and inhomogeneous flows  \cite{ChaDegMor24} and for the capillary flow  \cite{BelKho}, all  conditional on   \eqref{intro:ene-conv}.
In the mean-convex setting, De Philippis--Laux \cite{DePLau} proved the validity of \eqref{intro:ene-conv}, thereby removing the energy-convergence hypothesis in this case. 
In the recent contribution  \cite{LauPoi}, the authors considered a modified minimizing-movement scheme based on a more regularizing distance proxy, and obtain  unconditional convergence towards  De Giorgi varifold solutions to \eqref{intro:MCF}, as introduced in  \cite{HenLau}, in dimension two and three. 

In the present two-phase setting, where the comparison principle is available, it has been recognized by Chambolle \cite{Chambolle} that the scheme \eqref{intro:scheme} can be linked to the level set method of \cite{EvaSpr,CheGigGot}. 
Assuming the non-fattening condition, by results from the fundamental works \cite{EvansSpruckIV} and \cite{Ilmanen94} it is then possible to show that the Kuratowski/Hausdorff limit of discrete flows is a (unit density) Brakke flow. Again, this is  a conditional result  as it requires the non-fattening condition. Moreover, even along a common subsequence, Kuratowski limits in general only contain flat flows:  the strict inclusion is linked again to the disappearance of boundary components in the $L^1$-limit.  
Important classes of initial data for which non-fattening is known include star-shaped \cite{Soner93} and mean-convex sets  \cite{EvaSpr,Whi00}. We also mention  \cite{LauUll24} which shows that  almost every level set of level sets solutions are BV solutions.

Whether it is possible to prove the validity of a suitable evolution law for the limits of discrete flows originating from  the original 
scheme \eqref{intro:scheme} without assuming convergence of the perimeters \eqref{intro:ene-conv} or non-fattening has remained an open problem.

\subsection{Generalized BV flows and main result}

A different perspective is instead to look at the  associated  varifold measures and consider their limit in the varifold sense. It is not difficult to see that, denoting $\{E_t\}_{t \ge 0}$ a flat flow and $\{\mu_t\}_{t \geq 0}$ the mass measures corresponding to the limit varifolds $V_t$,  it holds 
\[
    |D\chi_{E_t}|\leq \mu_t\,.
\]
Suppose that the validity of a Brakke-type evolution law can be proved for the $V_t$'s. Then one may ask whether the mean curvature vector of the evolving varifolds also represents, in a suitable sense, the normal velocity field responsible for the motion of the (reduced) boundaries of $E_t$. This viewpoint is precisely captured by the notion of \emph{generalized BV flow}
introduced by Tonegawa and the second-named author in \cite{StuTon}. A generalized BV flow
is, roughly speaking, a one-parameter family of finite-perimeter sets $\{E_t\}_{t\ge 0}$ so that there exists a Brakke flow $\{V_t\}_{t\ge 0}$ with
\[ 
    |D\chi_{E_t}|\leq \|V_t\| \quad \text{for a.e. }t\ge0\,,
\]
and such that the sets $E_t$ are a BV solution in the spirit of \cite{LS}, with velocity field given by the mean curvature vector of the varifolds $\mathbf H_{V_t}$ instead of the mean curvature vector $\mathbf H_{E_t}$ of the boundaries of $E_t$. This philosophy has a
direct antecedent in Ilmanen's notion of \emph{enhanced motion}
\cite{Ilmanen94}, where a space-time \emph{undercurrent} is coupled with a
Brakke flow (\emph{overflow}) whose mass dominates the slices of the current. The
generalized BV formulation provides a 
realization of this principle, in which the coupling is expressed directly
through the transport identity for the evolving sets.

Generalized BV flows have recently appeared as a natural meeting point of
several approximation theories for mean curvature flow.
Tonegawa and the second-named author constructed such flows through their multiphase Brakke
approximation \cite{StuTon}, building on the existence theory of
Kim--Tonegawa \cite{KimTon}. Tashiro
 proved the generalized BV property for Ilmanen's elliptic regularization  \cite{Ilmanen94} and Allen-Cahn limits \cite{Tas24,Tas25}, and subsequently gave 
abstract conditions under which
an $L^2$-flow (as defined in  \cite{MugRog}) coupled with  evolving sets  satisfies the generalized BV law  \cite{Tas26}.  

Brakke flows, level-set flows, and generalized BV flows retain different
information: the level-set formulation has a powerful comparison structure
but does not record multiplicity, while the Brakke measure retains
interfacial mass but by itself does not determine the evolution of the
enclosed sets. The generalized BV formulation couples precisely these
latter two pieces of information.
 
The main result of this paper shows the validity of the generalized BV property for the solutions selected by the
original minimizing movements scheme of \cite{ATW,LS}. More
precisely, let  $\{E_t^{h}\}_{t\ge 0}$ be a sequence of discrete flows converging as $h\to 0$ in
$L^1$ to a flat flow $E_t$. Consider the associated multiplicity-one varifolds
\[
    V_{E_t^{h}}
    :=
    \mathcal{H}^n\llcorner\partial^*E_t^{h}
    \otimes \delta_{T_x\partial^*E_t^{h}}\,.
\]
It can be proved that, after extraction of a subsequence, these varifolds converge weakly to a family of varifolds $\{V_t\}_{t\geq0}$. More precisely, there exists a family $\{V_t\}_{t \geq 0}$ so that $V_t$ is integral for a.e. $t$ and
\begin{align*}
    \|V_{E_t^h}\| &\rightharpoonup \|V_t\| \quad \mbox{as Radon measures in $\R^{n+1}$, for every $t \geq 0$}\,,\\
    V_{E_{t^h}} \otimes dt &\rightharpoonup V_t \otimes dt \quad \mbox{as measures on $\R^{n+1} \times G(n+1,n) \times (0,\infty)$}\,,
\end{align*}
where $G(n+1,n)$ is the Grassmannian of (unoriented) $n$-planes in $\R^{n+1}$.

We can now provide an informal statement of our main result.

\begin{theorem*}
Let $1\leq n\leq 6$, and let the initial datum $E_0 \subset \R^{n+1}$ be any bounded set of finite perimeter with $\mathcal{L}^{n+1}(\partial E_0)=0$. Then, for every flat flow $E_t$ obtained from the
minimizing-movements scheme and varifolds  $V_t$ obtained as above, the couple $(\{E_t\},\{V_t\})$ is a generalized BV flow. Specifically,  $\{V_t\}_{t\geq0}$ is an integral Brakke flow with
\begin{align}
|D\chi_{E_0}| = \|V_0\| &= \lim_{t \to 0^+} \|V_t\|\,, \label{intro:initial} \\
 |D\chi_{E_t}| &\le \|V_t\| \quad\qquad\text{for a.e. }t>0\,, \label{intro:domination}\\
 \partial\llbracket E_t \rrbracket &= [V_t] \quad\quad\quad \,\text{in the sense of currents $\mathrm{mod}(2)$, for a.e. $t > 0$}\,, \label{intro:comp}
\end{align}      
and the sets $E_t$ satisfy the following evolution law:  for every $0\leq s<t$ and every smooth compactly supported
test function $\phi=\phi(x,\tau)$,
\begin{equation}\label{intro:gen-BV}
    \int_{E_t}\phi(x,t)\,dx
    -
    \int_{E_s}\phi(x,s)\,dx
    =
    \int_s^t\int_{E_\tau}
    \partial_\tau\phi(x,\tau)\,dx\,d\tau
    +
    \int_s^t\int_{\partial^*E_\tau}
    \phi(x,\tau)\,
    \mathbf{H}_{V_\tau}\cdot\nu_{E_\tau}\,
    d\mathcal{H}^n\,d\tau \,,
\end{equation}
where $\mathbf H_{V_t}$ is generalized mean curvature vector of the varifold $V_t$.
\end{theorem*}
This theorem provides, to the best of our knowledge, the first unconditional
consistency result for the original minimizing movements scheme for general admissible initial data. It resolves, in the
natural generalized BV formulation, the longstanding obstruction caused by
the possible loss of perimeter along the discrete approximations. Rather
than imposing convergence of the perimeters, the theorem retains the lost
interfacial mass in the limiting Brakke measure. It also avoids the conditional results holding for level-set solutions, where the absence of  fattening  is required. The conclusion therefore
applies to every flat flow selected by the scheme, without assuming
convergence of the energies, uniqueness of the discrete minimizers, or
absence of fattening. 

The generalized BV conclusion is not merely a technical weakening of the
Luckhaus--Sturzenhecker formulation. Since interface cancellations may happen in the limit, 
ordinary BV solution cannot in general record all the interfacial mass
generated by the approximation. The pair $(\{E_t\},\{V_t\})$ is therefore the
natural object selected by the scheme: $E_t$ describes the limiting sets,
whereas $V_t$ retains the limiting interfacial measure, including multiplicity invisible to $L^1$ convergence, and evolves by
Brakke's inequality. The relationship between the two pieces of information is expressed through \eqref{intro:initial}-\eqref{intro:gen-BV}. \eqref{intro:initial} states that the initial datum, intended in the sense of perimeter measure, is correctly captured by the mass measure of the initial varifold in the Brakke flow, and that the masses of varifolds in the Brakke flow are continuous at the initial time. \eqref{intro:domination} states that the varifold masses dominate the flat flow, so that the former cannot disappear so long as the latter is non-trivial. \eqref{intro:comp} states that the varifolds of the Brakke flow and the currents associated with the reduced boundaries of the flat flow are compatible as $\mathrm{mod}(2)$ flat chains in the sense of White \cite{Whi09}, so that $V_t=V_{E_t} + 2W_t$ for some integral varifold $W_t$, for a.e. $t > 0$: the multiplicity of $V_t$ is therefore odd on $\partial^* E_t$ and even on the complement. This justifies the intuition that \eqref{intro:ene-conv} is violated precisely if cancellations of the oriented boundaries happen. Finally, \eqref{intro:gen-BV} is the generalized BV identity, showing that the generalized mean curvature vector of the Brakke flow acts as the normal velocity field responsible for the motion of the reduced boundaries of the flat flow. 

The main theorem above is proved in Theorem \ref{thm:brakke} and Theorem \ref{thm:main}. Two corollaries are worth recording here: see Corollary \ref{cor:extinction} and Corollary \ref{cor:entropy-two}, respectively. First, by combining Brakke's inequality with the generalized BV identity one obtains, precisely as in \cite{GigaYamauchi93,StuTon}, a sharp lower bound on the extinction time of the Brakke flow; see the very recent \cite{cdcy} for related results on the upper bound, instead. Second, one can prove that the discrepancy between perimeter and varifold measures is zero if the perimeter measure of the initial datum has entropy $\lambda(|D\chi_{E_0}|) < 2$ in the sense of Colding-Minicozzi \cite{ColMin12}: in this case, \eqref{intro:ene-conv} holds precisely as a consequence of the approximation scheme rather than as an assumption, and the Brakke flow is unit density, thus subject to Brakke's partial regularity \cite{KasTon,STend}.

\subsection{Strategy of the proof}

The main novelty of the proof is a   derivation of Brakke's inequality by duality and 
directly from the dissipation of the original scheme of \cite{ATW,LS}.

As previously mentioned, past consistency results either assume \eqref{intro:ene-conv} or   rely on additional interpolations to recover the correct speed in the limit. This latter 
variational construction is usually referred to as the De Giorgi interpolant \cite{DeG,AmbGigSav}.   Such variational derivations of
Brakke-type inequalities  appear for instance in the analysis of threshold dynamics
by Laux--Otto \cite{LauOtt},  conditional to an energy convergence assumption similar to \eqref{intro:ene-conv}.

The mechanism  we use  is
different.  Rather than attempting to identify
the discrete velocity directly, we fix a smooth vector field $b$ and compare
the new iterate $E_k^h$ with the artificial evolution obtained by
transporting the preceding iterate $E^h_{k-1}$ for time $h$ along the flow
generated by $b$. This produces an inequality for every smooth $b$. After
passing to the varifold limit, a density argument allows us to optimize over
vector fields $b$. Completing the square then
identifies the optimal normal component with the generalized mean curvature
$\mathbf{H}_{V_t}$ and yields Brakke's inequality. We remark that the use of short smooth deformations is reminiscent of the discrete
deformation arguments underlying Brakke's original existence theory and the
construction of Kim--Tonegawa \cite{Bra,KimTon}. 

Our  two main  new technical improvements make this argument possible. The first one is  a   curvature estimate for every
discrete minimizer, obtained from the stability of the incremental problem. This provides the geometric control needed to employ normal coordinates at the natural scale $\sqrt h$. This is where we need to assume the dimension restriction $1 \leq n \leq 6$, which guarantees that the discrete minimizers are sufficiently smooth. The second is a common calibration field which simultaneously calibrates all the minimizers in \eqref{intro:scheme}.
The calibration is  used    to derive a localized comparison inequality that allows one to recover the correct quadratic dissipation without  introducing a different   interpolation. Both results appear to be of independent interest.

We now describe these  two technical  ingredients in more detail. The first is the
curvature estimate
\[
    \|A_{E_k^h}\|_{L^\infty(\partial E_k^h)}
    \leq
    \frac{C}{\sqrt{h}}.
\]
It is obtained from the stability inequality satisfied by each discrete
minimizer. If the estimate failed, a point-picking argument and a blow-up at a point of maximal curvature
would produce a complete stable minimal hypersurface with Euclidean area
growth and nonzero normalized curvature. The stable Bernstein theorem
\cite{Bel} excludes such a hypersurface when $n\leq 6$. Combined with the
density and local comparison estimates for the incremental problem, the
curvature bound also yields an interior and exterior ball condition at
scale $\sqrt{h}$, which, to the best of our knowledge have never been established before in this field.
These estimates are related in spirit to the proof of  density estimates via stability of
\cite{Jul24} and to the analysis of the ball condition for smooth flows
developed in \cite{JulNii24}. 

The second ingredient is a  calibration  for minimizers of the incremental problem.
Chambolle's level-set reformulation \cite{Chambolle} of the iterative  scheme \eqref{intro:scheme} associates to the
preceding iterate $F$ a function $u$ and a vector field $z$  satisfying
\[
    |z|\leq 1,
    \qquad
    \operatorname{div}z=\frac{u-d_F}{h},
    \qquad
    (z,Du)=|Du|\,,
\]
where $d_F$ is the signed distance from $\partial F$ (chosen so that it is negative at points in $F$), and where $(\cdot,\cdot)$ has to be understood in the sense of Anzellotti pairing. It is known \cite{Chambolle} that every minimizer $E$ of \eqref{intro:scheme} satisfies 
\[
\{u<0 \}\subseteq E\subseteq \{u\le 0\}\,,
\]
possibly with strict inequalities when fattening happens (that is, when $|\{u=0\}|>0$).
The crucial observation we make is that the \emph{same} field $z$ calibrates every
minimizer of the   incremental problem \eqref{intro:scheme}, including when the zero level set
of $u$ has positive measure.   This calibration yields a
localized comparison inequality  and it  replaces the
  estimates supplied by employing different variational interpolants.

We next explain schematically how these ingredients recover the sharp
quadratic coefficient in Brakke's inequality. Let
\[
    F=E^h_{k-1},
    \qquad
    E=E_k^h,
    \qquad
    G=\Phi_h^b(F),
\]
where $\Phi_h^b$ denotes the flow generated by $b$ at time $h$. By duality, the common calibration and the minimality of $E$  yield 
\[
    P(E)-P(F)
    \leq
    -\frac{1}{h}D(E,F)
    +\frac{1}{h}D(G,F)
    +h\,\delta P(F)[b]
    +\mathrm{Err}_k^h,
\]
where the errors are summable and vanish in the limit.   {This inequality holds no matter the vector field $b$. Moreover, thanks to the calibration we can easily derive a localized version of the previous inequality.} The dissipation cost of
the artificial motion $G$ can be computed in normal coordinates:
\[
    \frac{1}{h}D(G,F)
    =
    \frac{h}{2}
    \int_{\partial F}
        |b\cdot\nu_F|^2
        \,d\mathcal{H}^n
    +o(h).
\]
This provides one half of the quadratic term. The other half comes from the dissipation, as it is classical in the minimizing-scheme literature. Using the Euler--Lagrange equation together
with the tubular estimates  implied by the curvature  bounds, one
obtains
\[
    -\frac{1}{h}D(E,F)
    \leq
    h\,\delta P(E)[b]
    +
    \frac{h}{2}
    \int_{\partial F}
        |b\cdot\nu_F|^2
        \,d\mathcal{H}^n
    +o(h).
\]
The two terms with coefficient $1/2$ therefore combine into the full
quadratic contribution.

Carrying out this argument for the weighted quantities entering Brakke's
inequality gives, up to  errors, a discrete dual estimate of the form
\[
\begin{split}
    &\int_{\partial^*E_k^h}
        \varphi\,d\mathcal{H}^n
    -
    \int_{\partial^*E^h_{k-1}}
        \varphi\,d\mathcal{H}^n
    \\
    &\quad\leq
    h\Bigg[
        \delta P(E_k^h)[\varphi b]
        +\delta P(E^h_{k-1})[\varphi b]
        -\delta P(E_k^h)[\nabla\varphi]
        +
        \int_{\partial^*E^h_{k-1}}
            \varphi
            |b\cdot\nu_{E^h_{k-1}}|^2
            \,d\mathcal{H}^n
    \Bigg] + \mathrm{Err}^h_k
\end{split}
\]
The corresponding time-dependent inequality also contains the discrete
increments in time  of the test function.

After summation over the time steps and passage to the varifold limit, the first-variation terms combine into Brakke's term
\[
\int \varphi(\cdot,t) d\|V_t\|  - \int \varphi(\cdot,s) d\|V_s\|  \le  \int_s^t\int \Big( \partial_\tau \varphi + \nabla\varphi\cdot \mathbf H_{V_\tau} +\varphi(|b^\perp|^2- 2\mathbf H_{V_\tau}\cdot b  )\Big) d\|V_\tau\|d\tau .
\]
One can then complete the square 
\[
|b^\perp|^2- 2\mathbf H_{V_\tau}\cdot b = \left| \mathbf{H}_{V_t} - b^\perp \right|^2 - |\mathbf H_{V_t}|^2\,,
\]
infimize over (smooth) $b$ and use a density argument to recover Brakke's inequality with the sharp curvature
coefficient.

It remains to couple the resulting Brakke flow $\{V_t\}_{t\geq 0}$ with the
flat flow $\{E_t\}_{t\geq 0}$. 
The discrete dissipation bounds, combined with a scalar version of the preceding dual estimate, give absolute continuity of the space-time evolution of $E_t$ with respect to the limiting varifold measure. The generalized BV criterion of \cite{Tas26} then identifies the normal velocity of the phases with $\mathbf H_{V_t}\cdot\nu_{E_t}$ and yields the transport identity \eqref{intro:gen-BV}. Thus the same discrete approximations simultaneously produce
the Brakke evolution of the surviving interfacial measure and the BV
evolution of the underlying sets, while allowing for interfacial mass
which disappears in the $L^1$ limit.

\smallskip

The paper is structured as follows. In Section \ref{sec:prelim} we recall some preliminary estimates for discrete flows and show the curvature estimates. Then, in Section \ref{sec:Brakke} we study the calibration provided by Chambolle's scheme and deduce  Brakke's inequality. In the final  Section \ref{sec:main} we show that flat flows are generalized BV and deduce some corollaries.

\section{Preliminary estimates}\label{sec:prelim}
\subsection{Definition of the scheme and standard estimates}

Throughout, the ambient space is $\mathbb R^{n+1}$ and, unless
otherwise stated, $1\leq n\leq6$. The Lebesgue measure $\mathcal L^{n+1}$ is simply denoted $|\cdot|$, whereas $P(\cdot)$ denotes the perimeter functional, defined by
\[
P(F) := \sup\left\lbrace \int_F \diver X \, dx \, \colon \, X \in C^1_{c}(\R^{n+1};\R^{n+1}) \, \text{ s.t. } \|X\|_{C^0} \leq 1 \right\rbrace\, .
\]
If $U \subset \R^{n+1}$ is an open set, the relative perimeter $P(F;U)$ is defined with the obvious modification of the formula above. Given a set $F\subset\mathbb R^{n+1}$, the symbol $d_F$ will denote the signed distance function associated with $F$, defined by
\[
    d_F(x) := \mathrm{dist}(x,F) - \mathrm{dist}(x,\R^{n+1}\setminus F)\,,
\]
where, as usual,
\[
\mathrm{dist}(x,G) := \inf\left\lbrace |x-y|\,\colon\, y \in G\right\rbrace\,.
\]
In particular, $d_F(x) \geq 0$ if $x \notin F$ and $d_F(x) \leq 0$ if $x \in F$. 

The initial datum of the flow is a bounded set of finite perimeter $E_0 \subset \R^{n+1}$. We fix once and for all a bounded pointwise representative of $E_0$, so that the function $d_{E_0}$ can be given a meaning. In order to avoid trivial vanishing \footnote{For  $E_0=B_1\setminus \mathbb Q^{n+1}$, it holds $d_{E_0}\ge 0$ everywhere and thus the unique minimizer to \eqref{scheme} is the empty set. In particular, the flow is $E_t=\emptyset$ for all $t > 0$.}, we require additionally that $|\partial E_0|=0$, where $\partial$ denotes, as usual, topological boundary. For a set $F$ of (locally) finite perimeter, $\partial^*F$ denotes  reduced boundary; see \cite{Maggi}.

The minimizing movement scheme introduced in \cite{ATW,LS}, and further
studied for instance in \cite{MSS,Jul24}, is defined recursively by $E^h_0=E_0$ and, for $1\le k\in \N$
\begin{equation}\label{scheme}
 E^h_k\in\operatorname*{argmin}_{G}
 \left\{
 P(G)+\frac1hD(G,E^h_{k-1})
 \right\},
 \qquad
 D(G,F):=\int_{G\triangle F}|d_F|\,dx \,.
\end{equation}
In the above formula, the minimization is taken among sets $G$ of finite perimeter. For $k \geq 2$, as a result of the classical regularity theory for perimeter $(\Lambda,R)$-minimizers in low dimensions (see Lemma~\ref{lem:reg} below), we may always fix a unique open representative of $E^h_{k-1}$ so that $\partial^*E^h_{k-1}=\partial E_{k-1}^h$ and the latter is compact and of class $C^{2,\alpha}$. This choice of representative is set in place when computing $d_{E_{k-1}^h}$ for every $k \geq 2$. The number $h>0$ is the time discretization parameter, and $G \triangle F:=(G \setminus F) \cup (F \setminus G)$ is the symmetric difference between $G$ and $F$. We typically refer to $D(G,F)$ as the \emph{dissipation} between $G$ and $F$. Existence of minimizers is classically obtained via direct methods, see \cite{LS}; we remark that minimizers may not be unique.
Up to a constant independent of $G$, the minimizers in
\eqref{scheme} coincide with  the minimizers of
\[
 \mathcal J_F(G):=
 P(G)+\frac1h\int_Gd_F\,dx\,,
\]
since, as one can immediately verify,
\[
\int_{G \triangle F} |d_F| \, dx = \int (\chi_G - \chi_F) \, d_F\, dx = \int_G d_F \, dx - \int_F d_F\, dx\,.
\]

If $F=\emptyset$, we set the next iterate equal to $\emptyset$.
Accordingly, once an iterate is empty, all subsequent iterates are empty. 

\begin{definition}
    Consider  a bounded  set of finite perimeter $E_0$ as initial  datum and the time discretization parameter $h>0$.  The  \emph{discrete flow} starting from $E_0$ and with parameter $h$ is defined as $E^h_t:=E^h_{[t/h]}$ (where $[x]$ is the integer part of $x$) as long as the iterates are not empty,  with $E^h_k$ minimizing \eqref{scheme}. After the first time $t $ where $E^h_t=\emptyset$ we set  $E^h_s\equiv \emptyset$ for all $s\ge t$.
\end{definition}

For $\varphi \in C^0(\R^{n+1};\R_+)$, we define the weighted perimeter and dissipation
\[
 P_\varphi(G):=
 \int_{\partial^*G}\varphi\,d\mathcal H^n,
 \qquad
 D_\varphi(G,F):=
 \int_{G\triangle F}\varphi|d_F|\,dx\,.
\]

Our sign convention for the scalar mean curvature $H_E$ of a set $E$ with boundary of class $C^{2,\alpha}$ is
\[
 H_E=\diver_\tau\nu_E\,,
\]
where $\diver_\tau$ denotes tangential divergence, and $\nu_E$ is the outer normal vector field to $E$. More generally, if (the unit density varifold associated with) a set of finite perimeter $E$ has bounded and absolutely continuous first variation with respect to perimeter, then for any $X \in C^1_c(\R^{n+1};\R^{n+1})$ it holds
\begin{equation}\label{firstvar}
 \delta P(E)[X]
 =
 \int_{\partial^* E}\diver_\tau X\,d\mathcal H^n
 =
  -\int_{\partial^* E}\mathbf H_E\cdot X\,d\mathcal H^n=\int_{\partial^* E}H_E\,X\cdot\nu_E\,d\mathcal H^n \,,
\end{equation}
where $\mathbf H_E$ is the generalized mean curvature vector of $E$ and $H_E := -\mathbf H_E \cdot \nu_E$ is the generalized scalar mean curvature.

\medskip

We recall some standard estimates holding along the discrete flows. We refer to \cite{LS,MSS} for their proofs.

\begin{lemma}\label{lem:prel-est}
Let $E_0$ be a bounded set of finite perimeter. There exist dimensional constants $C>c>0, \gamma>1$ such that the following estimates
hold for $k\ge 1$:
\begin{align}
 P(E^h_{k})+\frac1hD(E^h_{k},E^h_{k-1})&\leq P(E^h_{k-1})\,,                                     
\nonumber\\
 \sup_{\partial E^h_{k}}|d_{E^h_{k-1}}|&\leq C\sqrt h\,,                              \nonumber\\
 \min\bigl\{|E^h_{k}\cap B_r(x)|\,,\ |B_r(x)\setminus E^h_{k}|\bigr\}
 &\geq cr^{n+1}\,,
 \qquad \, x\in\partial E^h_{k},\quad 0<r\leq\gamma\sqrt h\,,                   
 \label{density}\\
 cr^n\leq P(E^h_{k};B_r(x))&\leq Cr^n,
 \qquad \quad x\in\partial E^h_{k},\quad 0<r\leq\gamma\sqrt h\,,                   
 \label{area}\\
 |E^h_{k+1}\triangle E^h_{k}|
 &\leq C\left(
 \ell P(E^h_{k+1})+\frac1\ell D(E^h_{k+1},E^h_{k})
 \right),
 \qquad 0<\ell\leq\gamma\sqrt h\,,                         
\nonumber\\
 h\int_{\partial E^h_{k}}|H_{E^h_{k}}|^2\,d\mathcal H^n
 &\leq \frac ChD(E^h_{k},E^h_{k-1})\,.                                              
\nonumber
\end{align}
In particular, if $M:=P(E_0)$ and $Kh\leq T$, then
\begin{equation}\label{sums}
 \sup_{0\leq k\leq K}P(E^h_{k})\leq M\,,
 \quad
 \sum_{k=1}^KD(E^h_{k},E^h_{k-1})\leq hM\,,
 \quad
 \sum_{k=1}^Kh
 \int_{\partial E^h_{k}}|H_{E^h_{k}}|^2\,d\mathcal H^n
 \leq CM\,,
\end{equation}
and it holds 
\begin{equation}\label{volsum}
 \sum_{k=1}^K|E^h_k\triangle E^h_{k-1}|
 = |E_1^h \triangle E_0| + \sum_{k=2}^K|E^h_k\triangle E^h_{k-1}| \leq |E_1^h \triangle E_0|+ C(T+1)M\,.
\end{equation}
Moreover, the comparison principle holds and, in particular,  if $E_0\subseteq B_R$ for $R>0$ then $E^h_k\subseteq B_R$ for every $k\ge 1$.
\end{lemma}

\begin{remark}\label{rmk:c0}
    Let us show that, under the assumption $|\partial E_0|=0$, it holds 
    \[
    |E^h_1\triangle E_0|\to 0,\quad \text{as }h\to0\,.
    \]
    Indeed, consider $r>0$. We get  
    \[
    \begin{split}
        |E^h_1\triangle E_0| &= |(E^h_1\triangle E_0)\cap \{|d_{E_0}|\le r\}| + |(E^h_1\triangle E_0)\cap \{|d_{E_0}|> r\}| \\
        &\le |(E^h_1\triangle E_0)\cap\{|d_{E_0}|\le r\}| + \frac 1r \int_{(E^h_1\triangle E_0)\cap\{|d_{E_0}|>r\} } |d_{E_0}|\\
        &\le  |(E^h_1\triangle E_0)\cap\{|d_{E_0}|\le r\}| +\frac hr P(E_0)\,,
    \end{split}
    \]
    where in the last line we used the minimality estimate $D(E^h_1,E_0)\le hP(E_0)$. Sending first $h\to 0$ and then using the fact that $|\{|d_{E_0}|\le r\}|\to 0$ as $r\to 0$ since $|\partial E_0|=0$, we conclude.
\end{remark}

We also record the expression for the Euler-Lagrange equation and the stability inequality holding along the flow. We refer to \cite{LS, Jul24}.

\begin{lemma}\label{lem:reg}
If $1\leq n\leq6$, every nonempty iterate $E^h_k$ for $k \geq 1$ has compact
$C^{2,\alpha}$ boundary for every $0<\alpha<1$. Moreover, for every $k \geq 1$
\begin{equation}\label{EL}
 H_{E^h_{k}}=-\frac{d_{E^h_{k-1}}}{h}
 \qquad\text{on $\partial E^h_{k}$ }\,.
\end{equation}
For every $p\in\partial E^h_{k}$, every $0<r\leq\gamma\sqrt h$, and every
finite-perimeter set $G$ satisfying
$G\triangle E^h_{k}\Subset B_r(p)$, one has
\begin{equation}\label{quasi}
 P(E^h_{k};B_r(p))
 \leq P(G;B_r(p))
 +\frac C{\sqrt h}|E^h_{k} \triangle G|\,.
\end{equation}
Finally, denoting $A_E$ the second fundamental form of a (regular) set $E$,  for every $\zeta\in C^1(\partial E^h_k)$ it holds 
\begin{equation}\label{stab}
 \int_{\partial E^h_{k}}
 \left(
 |\nabla_\tau\zeta|^2-|A_{E^h_{k}}|^2\zeta^2
 \right)\,d\mathcal H^n
 \geq
 -\frac1h\int_{\partial E^h_{k}}\zeta^2\,d\mathcal H^n\,.
\end{equation}
\end{lemma}

The following property of discrete flows holds, and it implies H\"older continuity in time of their flat flow limits. 

\begin{lemma}
For every $T>0$ and every $0\leq s<t\leq T$,
\begin{equation}\label{discrete-time-holder}
 |E_t^h\triangle E_s^h|
 \leq CM\sqrt{|t-s|+h}+|E^h_h\triangle E_0|\,.
\end{equation}
The constant $C$ is independent of $h$, $s$ and $t$.
\end{lemma}

The preceding estimates give the usual compactness of the scheme.
Indeed, the sets $E^h_t$ are all contained in a common bounded set in $\R^{n+1}$, the perimeters are 
uniformly bounded, and \eqref{discrete-time-holder} gives   the uniform temporal estimate
 in $L^1$. Consequently, for every sequence $h_j\to 0$
and every corresponding choice of discrete minimizers, there exists a
subsequence and a family of finite-perimeter sets $\{E_t\}_{t\geq0}$
such that
\begin{equation}\label{flat-convergence}
\sup_{0\le t\le T} \|\chi_{E_t^{h_j}}- \chi_{E_t}\|_{L^1}\to 0, \quad \text{for every }T<+\infty. 
\end{equation}
Moreover,  for every $0\le s\le  t$ it holds  
\begin{equation} \label{eq:12Holder}
 P(E_t)\leq M,\qquad
 |E_t\triangle E_s|\leq CM\sqrt{|t-s|}\,.
\end{equation}

\begin{definition}[Flat flow]\label{def:flat-flow}
Any family $\{E_t\}_{t\geq0}$ obtained through
\eqref{flat-convergence} along a sequence $h_j\to0$
is called a flat flow starting from $E_0$.
\end{definition}

Lastly, we recall the notion of Brakke flow \cite{Bra}, referring to \cite{Ton} for additional material.

\begin{definition}[Brakke flow]\label{def:Brakke}
    A (measurable) family of $n$-dimensional varifolds $\{V_t\}_{t\ge 0}$ is an $n$-dimensional Brakke flow if:
    \begin{itemize}
        \item[i)] For a.e. $t\ge 0$, $V_t$ is  integral, with   bounded first variation  absolutely continuous with respect to $\|V_t\|$ and generalized mean curvature vector $\mathbf{H}_t$;
        \item[ii)] for all $T> 0$ and compact $K\Subset \R^{n+1}$ it holds 
        \[
            \sup_{t\in [0,T]}\|V_t\|(K) <+\infty,\qquad \int_0^T\int |\mathbf H|^2 \,d\|V_t\|\,dt<+\infty\,; 
        \]
        \item[iii)] for all $0\le s<t<+\infty$ and all test function $\phi \in C^1_c(\R^{n+1}\times[0,+\infty);[0,+\infty))$ it holds 
        \begin{equation}\label{eq:def-Brakke} 
            \int \phi(\cdot,t)  d\|V_t \| - \int \phi(\cdot,s)  d\|V_s \| 
            \le \int_s^t
            \int\left( \partial_\tau\phi+\nabla\phi\cdot \mathbf{H}_\tau- \phi |\mathbf H_\tau|^2 \right) d\|V_\tau\|\, d\tau \,.
        \end{equation}
    \end{itemize}
\end{definition}

\subsection{Curvature estimates via stability}
We recall the following definition of $(\Lambda,R)$-minimality of perimeter. We say that a set $F \subset \R^{n+1}$ of locally finite perimeter in an open set $U$ is a $(\Lambda,R)$-minimizer of the perimeter in $U$ if, for every ball $B_r(x) \Subset U$ with $0 < r < R$ and every set $G$ of locally finite perimeter such that
\[
F \triangle G \Subset B_r(x)
\]
one has
\begin{equation} \label{eq:almost-min}
   P(F;B_r(x)) \leq P(G;B_r(x)) + \Lambda |F \triangle G| \,.
\end{equation}
Notice that \eqref{quasi} implies that for every $k \geq 1$ any minimizer $E^h_k$ is a $(C\,h^{-1/2},\gamma h^{1/2})$-minimizer of the perimeter, modulo possibly reducing $\gamma$ so that \eqref{quasi} holds also for balls whose centers do not lie on $\partial E^h_k$. 

We recall the following standard compactness consequence of perimeter almost-minimality \cite{Maggi}: this will be applied to suitable rescalings of the sets $E^h_k$.

\begin{proposition} \label{prop:AM}
    Let $F_j$ be $(\Lambda_j, R_j)$-minimizers in expanding domains $U_j \uparrow \R^{n+1}$ with
    \begin{equation} \label{hp:AM}
        \Lambda_j \to 0\,, \qquad R_j \to \infty \,,
    \end{equation}
    and assume uniform local perimeter bounds and $0 \in \partial^*F_j$. Then, up to a subsequence
    \begin{equation} \label{th:AM}
        \chi_{F_j} \to \chi_{F_\infty} \quad \mbox{in $L^1_{loc}$}\,,
    \end{equation}
    where $F_\infty$ is locally perimeter minimizing in $\R^{n+1}$. Furthermore, the perimeter measures converge locally, and, upon choosing representatives of $F_j$ and $F_\infty$ so that $\partial F_j=\spt|D\chi_{F_j}|$ and $\partial F_\infty=\spt|D\chi_{F_\infty}|$, the boundaries $\partial F_j$ converge locally in Hausdorff distance to $\partial F_\infty$, and $0 \in \partial F_\infty$. Finally, if $n \leq 6$, the boundaries $\partial F_j$ converge locally as $C^{1,\alpha}$ graphs to $\partial F_\infty$ near every point of $\partial F_\infty$. 
\end{proposition}

By a rescaling procedure, we then show  curvature  estimates for the evolving sets $E^h_k$.

\begin{proposition}[Curvature estimate]\label{prop:curvature}
For $1\le n\le6$, every  iterate $E^h_k$ for $k\ge 1$ of \eqref{scheme} satisfies
\begin{equation*}
 \|A_{E^h_k}\|_{L^\infty(\partial E^h_k)}\le C_n/ \sqrt h\,.
\end{equation*}

\end{proposition}
\begin{proof}
We drop the superscript $h$ for simplicity. Suppose the estimate fails. Since the boundaries are (uniformly) compact and $C^2$, choose
\[
 \lambda_j=\max_{\partial E_j}|A_{E_j}|=|A_{E_j}|(p_j)\,,
\qquad  \lambda_j\sqrt{h_j}\longrightarrow\infty\,.
\]
Define $\widetilde E_j=\lambda_j(E_j-p_j)$. Then
\begin{equation}\label{normalize}
 |A_{\widetilde E_j}|\le1,\qquad |A_{\widetilde E_j}|(0)=1\,.
\end{equation}
The rescaled sets $\widetilde E_j$ are $\left(C (\lambda_j \sqrt{h_j})^{-1}, \gamma\lambda_j\sqrt{h_j} \right)$-minimizers of the perimeter in $\R^{n+1}$, and $0 \in \partial \widetilde E_j$. Moreover, rescaling \eqref{area} yields that for every fixed $R$
\[
P(\widetilde E_j;B_R)\le CR^n \qquad \mbox{for all sufficiently large $j$}\,.
\]
We can then apply Proposition \ref{prop:AM}, and conclude the existence of a locally perimeter minimizing set $E_\infty$ with $0 \in \partial E_\infty$ and Euclidean perimeter  growth $P(E_\infty;B_R)\le C R^n$. Its boundary is smooth, with local  $C^{1,\alpha}$ graphical convergence of $\partial \widetilde E_j$ to $\partial E_\infty$.

Next,  rescale \eqref{EL} to see that the  sets $\widetilde E_j$ satisfy 
\[
 H_{\widetilde E_j}(y)=-\frac{d_{F_j}(p_j+y/\lambda_j)}{h_j\lambda_j},
 \quad
 |H_{\widetilde E_j}(0)|\le\frac{C}{\lambda_j\sqrt{h_j}},\quad
 \Lip(H_{\widetilde E_j})\le\frac1{h_j\lambda_j^2}\,,
\]
where $F_j$ is the predecessor of $E_j$, and where we have also used Lemma \ref{lem:prel-est} and the fact that $d_F$ is $1$-Lipschitz. Thus $H_{\widetilde E_j}\to0$ locally uniformly, with vanishing Lipschitz seminorm. The $C^{1,\alpha}$ convergence of the boundaries implies local graphicality over a common plane $\pi_0$, and assume without loss of generality that $\pi_0=\R^n\times\{0\}$. If the coordinates are labeled $y=(y',y_{n+1})$, we can then let $y_{n+1}=v_j(y')$ denote the function locally representing $\partial \widetilde E_j$, and therefore satisfying
\[
-\diver\left(\frac{\nabla v_j}{\sqrt{1+|\nabla v_j|^2}} \right)(y') = H_{\widetilde E_j}(y',v_j(y'))
\]
The $C^{1,\alpha}$ bounds for the graphs and the uniform Lipschitz bounds for $H_{\widetilde E_j}$ give, by Schauder estimates, uniform local $C^{2,\alpha}$ estimates, which in turn imply that $v_j$ converge to a limit $v_\infty$ representing $\partial E_\infty$ in $C^{2,\beta}$ for all $0 < \beta < \alpha$. Hence, we deduce that
\begin{equation}\label{limitnorm}
 H_{ E_\infty}=0\,,\qquad |A_{ E_\infty}|(0)=1\,.
\end{equation}
Since $E_\infty$ is locally perimeter minimizing, $\partial E_\infty$ is stable. The component $M_\infty$ through $0$ of $\partial E_\infty$ is then two-sided and stable; it is also complete and properly embedded, since it is a
connected component of the closed smooth hypersurface
$\partial E_\infty$, and it has Euclidean area growth. The stable Bernstein theorem (see, for instance, \cite{Bel}) implies then that $M_\infty$ is a hyperplane when $2\le n\le6$.  This contradicts \eqref{limitnorm}. The case $n=1$ follows straightforwardly from the Euler-Lagrange equation.
\end{proof}

Let us recall the notion of reach for a regular set. Given a bounded set $E$ of class $C^2$, we define the \emph{reach} of $\partial E$ by
\begin{equation} \label{def:reach}
\begin{split}
\operatorname{reach}(\partial E)
:=
\sup\Big \lbrace & r \ge 0: \mbox{every $y\in\R^{n+1}$ with $\text{dist}(y, \partial E) <r$} \\  & \mbox{has a unique nearest point on $\partial E$}\Big\rbrace\,.
\end{split}
\end{equation}
In particular, $d_E$ is $C^1$ on $\{|d_E|<\text{reach}(\partial E)\}$
and its gradient is determined by the unique nearest-point projection onto $\partial E$.

\begin{lemma}[Reach estimate] \label{lem:reach}
Let $1\le n\le6$. Assume that $\partial E$ is a compact $C^2$ boundary such that there exist a radius $r_0 > 0$ and positive numbers $K,L$ so that for every $x\in\partial E$ the following holds in $B_{r_0}(x)$:
\begin{enumerate}
\item $|A_E|\le K/r_0$;
\item \eqref{density}--\eqref{area} hold for all $r\le r_0$, with fixed constants;
\item  the set $E$ is a $(L/r_0,  r_0)$-minimizer of the perimeter in $B_{r_0}(x)$.
\end{enumerate}
Then, there exists $c>0$ depending only on the constants above  such that
\[
 \reach(\partial E)\ge c r_0\,.
\]

\end{lemma}

\begin{proof}
Assume the contradiction: there exist $E_j$ satisfying the assumptions with $r_0=r_{0,j}$ and  $\rho_j=\reach(\partial E_j)$ with $\rho_j/r_{0,j}\to0$ as $j\to +\infty$. By (1), the focal radius is at least $r_{0,j}/K$, so that, since $\rho_j = o(r_{0,j})$, for all sufficiently large $j$ $\rho_j$ is smaller than the focal radius. By definition of reach, then, there exists a point $m_j$ with $\text{dist}(m_j,\partial E_j)=\rho_j$ and  with two distinct nearest points $p_j,q_j\in\partial E_j$. Recenter and rescale   defining $\widetilde E_j=(E_j-p_j)/\rho_j$ and $\tilde m_j,\tilde p_j,\tilde q_j$ accordingly. For such sets $\reach(\partial \widetilde E_j)=1$, and the assumptions are satisfied with $r_0$ replaced by $r_{0,j}/\rho_j \to \infty$. In particular, the curvature bound becomes $|A_{\widetilde E_j}|\le K\rho_j/r_{0,j}\to 0$,  and 
\[
 \tilde p_j=0,\qquad |\tilde m_j|=1,\qquad |\tilde q_j-\tilde m_j|=1,\qquad |\tilde q_j|\le2.
\]
By assumption $(3)$, the sets are uniformly $(L\rho_j/r_{0,j},  r_{0,j}/\rho_j)$-minimizers, and thus Proposition~\ref{prop:AM}
produces, up to extraction of a (not relabeled) subsequence,  a complete smooth locally minimizing boundary $\partial E_\infty$ with local  graphical convergence in $C^{1,\alpha}$. The vanishing curvature bound implies that each limiting component is a hyperplane. 

Note that there can be only one hyperplane. Indeed, local perimeter finiteness makes any family of parallel components of $\partial E_\infty$ locally finite. Intersecting planes are also excluded by the regularity of $E_\infty$.   Consider two adjacent planes  bounding a slab occupied by $E_\infty$. Consider a cylinder of radius $R$ and height $L$, with base  parallel to the planes. It can be chosen to intersect only one slab. 
Removing the cylinder from $E_\infty$  decreases the perimeter by  $  cR^n$ 
and adds only the lateral perimeter $cLR^{n-1}$. For $R$ large enough this contradicts  local minimality. Hence the limit is a single plane.

Note that, in particular, the graphical convergence holds in every neighborhood of $0$. Therefore, for $j$ large enough, both nearest points $\tilde p_j, \tilde q_j$ lie  on one graph over this plane in a fixed cylinder. Assuming for simplicity that the limiting plane is $\{x_{n+1}=0\}$, we can parametrize the rescaled sets $\widetilde E_j$ in a fixed neighborhood $B_3\subseteq \{x_{n+1}=0\}$ as graphs $v_j: B_3\to \R$. By the convergence in $C^{1,\alpha}$ we get $\|v_j\|_\infty+\|\nabla v_j\|_\infty \to 0$, and in fact also 
\[
\|\nabla^2 v_j\|_\infty \leq C (1 + \|\nabla v_j\|_\infty^2)^{3/2} \|A_{\widetilde E_j}\|_\infty \to 0\,.
\]
As a consequence, on $B_3$ the squared distance from $\partial \widetilde E_j$ to $\tilde m_j$ has positive definite Hessian: writing $x=(x',x_{n+1})$ it holds 
\[
 D^2\bigl|(x',v_j(x'))-\tilde m_j\bigr|^2
 =2I+2\nabla v_j\otimes \nabla v_j+2(v_j-\tilde m_{j,n+1})\nabla^2v_j>0\,,
\]
for all $j$ sufficiently large, where we used $|\tilde m_j|\le1$.
Therefore, such function  cannot have two distinct interior critical points. Recalling that $\tilde p_j=0$ and $|\tilde q_j|\le 2$,  there are two nearest points to $\tilde m_j$ in $B_3$. Nearest points are critical for the squared distance function, a contradiction.
\end{proof}

As a direct corollary one gets the following curvature and reach estimates holding along all the iterates of the minimizing movements scheme. The proof is immediate from Proposition \ref{prop:curvature} and Lemma \ref{lem:reach} with $r_0=\sqrt{h}$.

\begin{corollary} 
Let $1\le n\le 6$, $h>0$. Every nonempty $E^h_k$, $k\ge 1$ satisfies
\begin{equation}\label{tubeest}
 \|A_{E^h_k}\|_\infty\le C/\sqrt h,
 \qquad \reach(\partial E^h_k)\ge c\sqrt h\,.
\end{equation}
In particular, $E^h_k$  satisfies the interior and exterior ball conditions with any radius smaller than $c\sqrt h$. The constants $C,c>0$ depend only on $n$.
\end{corollary}

 Lastly, we record the following standard  computations in normal cordinates.

\begin{lemma}[Tubular estimates]\label{lem:tubecalc}
If $F$ satisfies \eqref{tubeest}, then for $|s|<a\sqrt h$, with fixed sufficiently small $a>0$, the map $\Psi(p,s)=p+s\nu_F(p)$ is injective and, setting $F_s=\{d_F<s\}$, it holds 
\begin{align}
 &J_\Psi(p,s)=\prod_{i=1}^n(1+s\kappa_i(p)),
 &|J_\Psi(p,s)-1|\le C|s|/\sqrt h\,,\label{jac}\\
 &P(F_s)\le(1+C|s|/\sqrt h)P(F),
 &|F_s\triangle F|\le C|s|P(F)\,,\label{parallel}\\
 &\nabla d_F(\Psi(p,s))=\nu_F(p),
 &\|D^2d_F\|\le C/\sqrt h\quad\hbox{in  } \{|d_F|< a \sqrt{h}\} \,.\label{hessd}
\end{align}
\end{lemma}

The formulas above  follow straightforwardly 
using the area formula in normal coordinates and using the relation $\kappa_i(F_s)=\kappa_i(F)/(1+s\kappa_i(F))$, see for instance \cite{GilbargTrudinger}, where $\{\kappa_i(F)\}$ denote  principal curvatures.

\section{Brakke's inequality}\label{sec:Brakke}
The goal of this section is to first derive an approximate version of Brakke's inequality holding along the discrete flow, and then pass this property to the limit.

\subsection{Anzellotti's pairings}\label{sec:Anz}
We gather some technical results about \textit{divergence-measure pairings} in the sense of Anzellotti \cite{Anz}, see also  \cite{CraDec}. 
We recall the relevant notion only in the generality needed for the present paper.

Consider an open domain $\Omega\subset \R^{n+1}$ and  a vector field $z\in L^\infty(\Omega;\R^{n+1})$ with bounded divergence $\diver z\in L^\infty(\Omega)$. For every $u\in BV(\Omega)$ we can define the \textit{pairing} between $z$ and $Du$ as the distribution $(z,Du)$ acting on $C^\infty_c$ by
\[
    (z,Du)(\varphi)=-\int_\Omega u\varphi \diver z -\int_\Omega u z\cdot\nabla \varphi, \qquad  \forall \varphi\in C^{\infty}_c(\Omega).
\]
It turns out that such a pairing is a Radon measure satisfying $(z,Du)\ll |Du|$. Actually, it holds 
\begin{equation}\label{eq:pairing-bound}
 |(z,Du)|\leq \|z\|_{L^\infty(\Omega)}\,|Du|
\end{equation} 
in the sense of measures, 
\cite[Theorem 1.5]{Anz} and
\cite[Section 2.5]{CraDec}.

Given a set of finite perimeter  $E$, it is possible to define a notion of normal trace of the vector field $z$ on $\partial^* E$. For a $C^1$ open subset  $\Omega'\Subset \Omega$, the \textit{normal trace} of $z$ on $\partial \Omega'$ with respect to the outer normal $\nu_{\Omega'}$ is the distribution
\[
 [z,\nu_{\Omega'}](\varphi)=\int_{\Omega'} z\cdot \nabla \varphi + \int_{\Omega'}\varphi \diver z, \qquad  \forall \varphi \in C^\infty_c(\Omega).
\]
The normal trace is actually absolutely continuous with respect to $\mathcal{H}^n\llcorner\partial \Omega'$. The Radon-Nikodym derivative of $[z,\nu_{\Omega'}]$ with respect to $\mathcal{H}^n\llcorner\partial \Omega'$, still denoted for simplicity as $[z,\nu_{\Omega'}]$, satisfies 
\begin{equation}\label{eq:nor-tr-est}
    \|[z,\nu_{\Omega'}]\|_{L^\infty(\partial \Omega')}\le \|z\|_{L^\infty(\Omega)}.
\end{equation}
For a set of finite perimeter $E$,   the normal trace of $z$ on $\partial ^*E$ with respect to the outer normal $\nu_E$ can be  defined on every $C^1$ component as a consequence of De Giorgi's structure theorem \cite[Theorem 15.9]{Maggi}, see \cite{Anz} for details. In particular, consider the instance where  $E\subset\Omega'\Subset \Omega $ has finite perimeter, it holds 
\begin{equation}\label{eq:pairing-trace}
 (z,D\chi_E)
 =-[z,\nu_E]\,|D\chi_E|.
\end{equation}
Moreover, for every $\varphi\in C_c^1(\Omega)$, the Gauss--Green formula reads
\begin{equation}\label{eq:pairing-gauss-green}
 \int_E \varphi\,\diver z
 +\int_E z\cdot\nabla\varphi
 =\int_{\partial^*E}\varphi[z,\nu_E]\,
     d\mathcal H^n;
\end{equation}
see \cite[Theorem 5.1]{CraDec}. 
We shall also use the coarea formula for the pairing: as Radon measures it holds 
\begin{equation}\label{eq:pairing-coarea}
 (z,Du)=\int_{\mathbb R}(z,D\chi_{\{u>\lambda\}})\,d\lambda
 .
\end{equation}
Under the assumptions above this follows from
\cite[Proposition 2.7]{Anz}; see also
\cite[Theorem~4.2]{CraDec}.

\subsection{Common ROF calibration}
We use the level-set reformulation of the scheme \eqref{scheme} due to Chambolle \cite{Chambolle} in terms of a  Rudin-Osher-Fatemi \cite{RuOsFa:92} functional (ROF in short). All pairings below are Anzellotti pairings as recalled in Section \ref{sec:Anz}.

To avoid an infinite  ROF energy for $d_F$ for instance, we work on a  flat torus $\mathbb  T^{n+1}$. Up to rescaling we can assume $E_0\subseteq B_{1/8}$. By the comparison principle,  every discrete (and thus flat) flow $\{E^h_k\}_{k\in\N}$ starting from $E_0$ will stay inside $B_{1/8}$. Moreover, the periodic signed distance to  $E^h_k$ agrees with $d_{E^h_k}$ on $B_{1/4}\subseteq  [-\tfrac12,\tfrac12]^{n+1}$.  For small $h$, the zero sublevels below also lie in this larger ball. All set energies compared there are therefore the original Euclidean energies and thus the localization in $\mathbb T^{n+1}$ does not change the evolving sets.

We start by recalling  classical estimates for the ROF problem.

\begin{lemma}[ROF estimates]\label{lem:ROF}
Let $F\subset \mathbb T^{n+1}$. The (unique) minimizer
\[
 u=\mathop{\rm argmin}_{v\in BV\cap L^2}
 \left\{|Dv|(\mathbb T^{n+1})+\frac1{2h}\int|v-d_F|^2\right\}
\]
admits a calibration $z\in L^\infty(\mathbb T^{n+1};\R^{n+1})$ satisfying in $\mathbb T^{n+1}$
\begin{equation}\label{zconditions}
 |z|\le1,\qquad \diver z=(u-d_F)/h,\qquad (z,Du)=|Du|,
\end{equation}
with the last equality intended in the sense of Radon measures. Moreover,
\begin{equation}\label{zbound}
 \Lip(u)\le1,\qquad \|u-d_F\|_\infty\le C\sqrt h,
 \qquad \|\diver z\|_\infty\le C/\sqrt h.
\end{equation}
\end{lemma}

\begin{proof}
Convex  duality for the total variation functional  gives the calibration \eqref{zconditions}, see for instance \cite[Section 2.2]{ChaCascreNovPoc-intro}. 
By translation invariance, it is not difficult to see that  $u$ is $1$-Lipschitz, thus  $w=u-d_F$ is $2$-Lipschitz. Consider a maximum point  $w(x)=m>0$, then $w\ge m-\sqrt{h}$ on $B_{\sqrt{h}/4}(x)$. For $h$ small this ball is inside the injectivity radius of the flat torus. Integrating $w/h=\diver z$ over $B_{\sqrt{h}/4}(x)$, using Gauss-Green and  $|z|\le1$ gives $m-\sqrt{h} \le (1+C) \sqrt{h}$. Since  the same argument applies to $-w$, we deduce \eqref{zbound}.
\end{proof}
Note that $u\le0$ implies $d_F\le C\sqrt h$ by  \eqref{zbound}, which, for $h$ small,  shows that the  zero sublevel set of $u$ is well contained in $\mathbb T^{n+1}$.

We then show that the calibration provided by the ROF problem is a calibration for every minimizer of the incremental problem \eqref{scheme}. The argument  is inspired by the proof of \cite[Proposition 2.2]{Chambolle}.

\begin{lemma}[Calibration of every minimizer]\label{lem:allmin}
Let $u,z$ be as above.  If $E$ is any minimizer of \eqref{scheme} with $F$ replacing $E^h_{k-1}$, then 
\begin{equation}\label{sign}
 \{u<0\}\subseteq E\subseteq \{u\le0\}, \text{ up to }\mathcal L^{n+1}\text{-negligible set}
\end{equation}
\begin{equation}\label{trace}
 [z,\nu_E]=1\quad\HH^n\text{-a.e. on }\partial^*E\,,
\end{equation}
where \eqref{trace} is intended in the sense of normal traces. 
\end{lemma}

\begin{proof}
For a bounded set of finite perimeter  $A\subset \mathbb T^{n+1}$, note that  
\[
 P(A)-\int_A\diver z
 =\int_{\partial^*A}(1-[z,\nu_A])d\HH^n\ge0\,,
\]
where the the first identity follows from \eqref{eq:pairing-gauss-green} and  last inequality follows from \eqref{eq:nor-tr-est}.
By the coarea formula \eqref{eq:pairing-coarea}
it holds 
\[
(z,Du)=\int_\R (z,D\chi_{\{u>\lambda\}})d \lambda, \qquad |Du|=\int_\R|D\chi_{\{u>\lambda\}}|d \lambda\,,
\]
in the sense of measures. Since in general $|(z,D\chi_A)|\le |D\chi_A|$ as $|z|\le 1$, by the pairing \eqref{zconditions}
we get 
\[
0=|Du|-(z,Du)=\int_\R \Big( |D\chi_{\{u>\lambda\}}|- (z,D\chi_{\{u>\lambda\}}) \Big) d\lambda \ge 0,
\]
and thus the integrand is zero for a.e. $\lambda$. Since $|\{u=\lambda\}|=0 $ for a.e. $\lambda$, $D \chi_{\{u<\lambda\}} = - D \chi_{\{u>\lambda\}}$ and therefore 
\[
|D\chi_{\{u<\lambda\}}|=- (z,D\chi_{\{u<\lambda\}})\quad \text{for a.e. }\lambda.
\]
Thus, by Gauss-Green \eqref{eq:pairing-gauss-green} we deduce for a.e. $\lambda$
\[
\int_{\{ u<\lambda \}} \diver z=-(z,D\chi_{\{u<\lambda \}}) (\mathbb T^{n+1})=|D\chi_{\{u<\lambda\}}|(\mathbb T^{n+1})=P(\{u<\lambda \})\,.
\]
Choose levels $\lambda_j\nearrow0$ satisfying the previous. The characteristic functions $\chi_{\{u<\lambda_j\}}$ converge in $L^1(\mathbb T^{n+1})$ to that of $E_-:=\{u<0\}$. Since $\diver z\in L^\infty(\mathbb T^{n+1})$, lower semicontinuity gives
\[
 P(E_-)\le\liminf_jP(\{u<\lambda_j\})
 =\int_{E_-}\diver z\le P(E_-).
\]
Thus $P(E_-)=\int_{E_-}\diver z$. For $\cJ_F(G)=P(G)+h^{-1}\int_G d_F$, using the previous identity one gets
\begin{equation}\label{defect}
 \cJ_F(G)-\cJ_F(E_-)
 =P(G)-\int_G \diver z+\int_{G\cap\{u> 0\}}\tfrac uh
                  +\int_{G^c\cap\{u<0\}}(-\tfrac uh).
\end{equation}
The term $P(G)-\int_G \diver z$ and the other two terms on the right are nonnegative. Consider another minimizer $E$ of \eqref{scheme}. Both $E$ and $E_-$ are compactly contained in $\mathbb T^{n+1}$ so the energies are the same as the Euclidean counterparts. Minimality of $E$ makes the left hand side of \eqref{defect}, with $G=E$, non positive. Hence  every term vanishes. The last two give \eqref{sign}; the first gives \eqref{trace}.
\end{proof}

 We introduce the following notation: if $\{E^h_k\}$ is a discrete flow, we denote as $z^h_k$ a   calibration produced by Lemma~\ref{lem:allmin} with $F=E^h_{k-1}$, thus  relative to the $k$-th iteration of the scheme \eqref{scheme}. Moreover, along a discrete flow $E^h_k$ we define the quantity 
 \[
    S^h_k:=P(E^h_{k-1})-\int_{E^h_{k-1}}\diver z^h_k, \quad \text{for }k\ge 1\, .
 \]

By the very existence of the calibration $z$ we deduce the following almost-interpolation inequality, which will be crucial to show Brakke's inequality in the limit.

\begin{lemma} 
For $k\ge 1$ it holds 
\begin{equation}\label{Ssum}
 0\le S^h_k\le P(E^h_{k-1})-P(E^h_{k})-\tfrac1hD(E^h_{k},E^h_{k-1}),
 \qquad \sum_{k=1}^KS^h_k\le M\,.
\end{equation}
For any compact set of finite perimeter $G\subset B_{1/4}$ and  $\varphi\in C_c^1$, $\varphi\ge0$, it holds 
\begin{equation}\label{local}
 P_\varphi(E^h_{k})-P_\varphi(G)
 \le-\frac{D_\varphi(E^h_{k},E^h_{k-1})}h+\frac{D_\varphi(G,E^h_{k-1})}h
     +\int(\chi_{E^h_{k}}-\chi_G)z\cdot\nabla\varphi\,.
\end{equation}
\end{lemma}

\begin{proof}
Set for simplicity  $E=E^h_k$, $F=E^h_{k-1}$ and $z=z^h_k$.
Using the calibration and the identity $\diver z=(u-d_F)/h$ we get 
\[
 S_k=P(F)-P(E)-\tfrac 1h D(E,F)+\int(\chi_E-\chi_F)\tfrac uh.
\]
The last integral is non positive by \eqref{sign}; summing and using \eqref{sums} proves \eqref{Ssum}.
By $|z|\le 1$ and  the  Gauss-Green formula for Azellotti's pairings 
\begin{equation}\label{eq:cali-pb}
P_\varphi(G)\ge \int_{\partial^* G}\varphi[z,\nu_G] = \int_G \diver(\varphi z).
\end{equation}
Equality holds  for $G=E$ by \eqref{trace}. Subtract the formulas for $E,G$, use $h\diver z= u-d_F$ and 
\[
\int (\chi_E-\chi_G)\varphi u\le  \int_{\{u\le 0 \}\setminus G}\varphi u - \int_{G\cap\{u\ge 0\}} \varphi u  \le0
\]
to obtain \eqref{local}.
\end{proof}

\subsection{Towards Brakke's inequality }

For ease of notation, in this section we fix $k\ge 2$ and 
\[
    E=E^h_k,\quad F=E^h_{k-1}, \quad  S=S^h_k\,, 
\]
In particular,  both  $E,F$ satisfy  the curvature estimates \eqref{tubeest}.  Fix $\varphi\in C_c^2$, $\varphi\ge0$, and $b\in C_c^2(\R^{n+1};\R^{n+1})$. Constants may depend on these tests. All integrals without an indicated measure are Lebesgue integrals.

Given the vector field $b$, let $\Phi_h$ be the associated flow at time $h$. We recall that this is defined as $\Phi_h(x):=\Phi(x,h)$ with $\Phi(x,t)$ solving 
\[
    \partial_t \Phi(\cdot,t)= b(\Phi(\cdot,t)),\quad
    \Phi(\cdot,0)=id \,.
\]
We  apply  \eqref{local} with competitor $G=\Phi_h(F)$   to  obtain
\begin{equation}\label{eq:Brakke0}
 P_\varphi(E)-P_\varphi(F)
 \le 
 -\frac{D_\varphi(E,F)}h
 +\frac{D_\varphi(G,F)}h                    
+\int(\chi_E-\chi_G)z\cdot\nabla\varphi
 +P_\varphi(G)-P_\varphi(F)\,.    
\end{equation}
Note that $\Phi_h(F)$ lies in a  neighborhood of $F$ of size $h$, thus for all sufficiently small $h$ the curvature estimates of Proposition~\ref{prop:curvature} guarantee that we can employ a normal graph parametrization. The following lemmas  separately estimate the terms appearing in \eqref{eq:Brakke0} and link them to a discrete version of Brakke's inequality.

\begin{lemma}[Dissipation bound] \label{lem:diss-bd}
For $h$ sufficiently small, depending on $b$, it holds 
\begin{equation}\label{endpointineq}
 -\frac{D_\varphi(E,F)}h
 \le h\delta P(E)[\varphi b]
 +\frac h2\int_{\partial F}\varphi|b\cdot\nu_F|^2d\HH^n
 +C\bigl(D(E,F)+h^{3/2}P(F)\bigr)\,.
\end{equation}
\end{lemma}
\begin{proof}
Since $d_F=0$ on $\partial F$, Gauss--Green and the Euler-Lagrange equation \eqref{EL} give
\begin{equation}\label{gaussd}
 -h\delta P(E)[\varphi b] = \int_{\partial^* E} d_F \varphi \nu\cdot b - \int_{\partial^* F} d_F \varphi \nu\cdot b 
 =\int(\chi_E-\chi_F)
 \left(\varphi\nabla d_F\cdot b+d_F\diver(\varphi b)\right)\,.
\end{equation}
Write $\sigma=\operatorname{sgn}(d_F)$, so $\chi_E-\chi_F=\sigma|\chi_E-\chi_F|$ a.e.. Bounding  $\diver (\varphi b)\le C$ and $|\chi_E-\chi_F|\le 1$ we get 
\begin{equation}\label{signedbound}
 -h\delta P(E)[\varphi b]- \frac{D_\varphi(E,F)}h
 \le
 \int\varphi\left(\sigma\nabla d_F\cdot b-|d_F|/h\right)_+
       +C D(E,F)\,.
\end{equation}
The positive part lies in $\{|d_F|\le h\|b\|_\infty\}$. Since  $h\|b\|_\infty \ll\sqrt h$, for  $h$ small enough we can use normal coordinates $\Psi$ from Lemma \ref{lem:tubecalc} and \eqref{hessd} to write 
\begin{multline*}
        \int\varphi\left(\sigma\nabla d_F\cdot b-|d_F|/h\right)_+
= 
\int_{\partial F}\int_0^{h\|b\|_\infty} \varphi\circ \Psi \left( ( \nabla d_F\cdot b)\circ\Psi - \frac sh \right)_+ J_\Psi
\\
 +\int_{\partial F}\int^0_{-h\|b\|_\infty} \varphi\circ \Psi \left( -(\nabla d_F\cdot b)\circ\Psi + \frac sh \right)_+ J_\Psi \,.
\end{multline*}
Using Lemma \ref{lem:tubecalc}, we see that freezing the argument of    $b,\varphi$ contribute $O(h^2P(F))$, whereas the Jacobian \eqref{jac} contributes $O(h^{3/2}P(F))$.
Since $\nabla d_F(\Psi(p,s))=\nu_F(p)$, we get 
\begin{multline*}
        \int\varphi\left(\sigma\nabla d_F\cdot b-|d_F|/h\right)_+
= 
\int_{\partial F} \varphi(p)
 \int_0^\infty\bigl(b(p)\cdot\nu_F(p)-s/h\bigr)_+d s \,d\mathcal H^{n}_p
 \\
 +\int_{\partial F}\varphi(p) \int_0^\infty\bigl(-b(p)\cdot\nu_F(p)-s/h\bigr)_+d s\,\,d\mathcal H^{n}_p + O(h^{3/2}P(F))\,.
\end{multline*}
Since an explicit computation yields 
\[
\int_0^\infty\bigl(b(p)\cdot\nu_F(p)-s/h\bigr)_+d s
 + \int_0^\infty\bigl(-b(p)\cdot\nu_F(p)-s/h\bigr)_+d s= \frac h2 |b(p)\cdot \nu_F(p)|^2,
\]
we conclude  \eqref{endpointineq}.
\end{proof}

In the proof above, keeping track of the sign of $d_F$ by  $\sigma$ is crucial. 
Replacing $\sigma\nabla d_F\cdot b$ by $|\nabla d_F\cdot b|$ in \eqref{signedbound} would double the coefficient in the   proof of the approximate version of Brakke's inequality holding at the discrete level.

\begin{lemma}[Calibration error] 
Let $z$ be   a calibration for $E$ produced by Lemma \ref{lem:allmin}. For $0<r<a\sqrt h$, with $a$ as in Lemma \ref{lem:tubecalc}, it holds 
\begin{equation}\label{tubedefect}
 \int_{\{|d_F|<r\}}|z-\nabla d_{F}|^2
 \le C\left(rS +\frac{r^2}{\sqrt h}P(F)\right)\,.
\end{equation}
For $ R\ge 1$ with $hR<a\sqrt h$, and for $K,T$ so that $Kh\le T$,  it holds 
\begin{equation}\label{asum}
 \sum_{k=2}^{K}\int_{E^h_k\triangle E^h_{k-1}}
 |z^h_k-\nabla d_{E^h_{k-1}}|
 \le C_T\left(R^{-1}+h^{1/2}R^{1/2}+h^{1/4}R\right)\,,
\end{equation}
where $C_T$ depends on $M,T$.
Moreover,
\begin{equation}\label{eq:transp-term-old}
 \left|\int(\chi_E-\chi_F)z\cdot\nabla\varphi
           +h\delta P(E)[\nabla\varphi]\right|
 \le C\left(\int_{E \triangle F}|  z-\nabla d_F |+D(E,F)\right)\,.
\end{equation}
\end{lemma}

\begin{proof}
For $F_s=\{d_{F}<s\}$ and $|s|<a\sqrt h$, the graphical computations \eqref{parallel}, the  calibration bounds \eqref{zbound} and Gauss--Green imply
\begin{equation}\label{cali-err}
     0\le P(F_s)-\int_{F_s}\diver z
 \le S+C|s|P(F)/\sqrt h\,.
\end{equation}
Now consider a mollification $\rho_\varepsilon\ast z:=z_\varepsilon$. From $|z_\varepsilon-\nabla d_F|^2\le2(1-z_\varepsilon\cdot\nabla d_F)$, the   estimate \eqref{hessd} implying $d_F\in C^{1,1}(\{|d_F|<a\sqrt{h}\})$ and $\nu_{F_s}=\nabla d_F\llcorner \partial F_s$, by coarea 
\[
\frac12\int_{\{|d_F|<r\}}|z_\varepsilon -\nabla d_F|^2\le \int_{-r}^r \Big( P(F_s) - \int_{\partial F_s} z_\varepsilon\cdot \nabla d_F\Big) ds = \int_{-r}^r \Big( P(F_s) - \int_{ F_s}  \diver z_\varepsilon\Big) ds\,.
\]
We then deduce   \eqref{tubedefect} letting first $\varepsilon\to 0$ and then  using \eqref{cali-err}. 

Restore the notation $E^h_k=E,E^h_{k-1}=F$ and take  $Kh\le T$.
In the region $\{ |d_{E^h_{k-1}}|\ge hR\}$, we have by Chebyshev the obvious bound
\begin{equation}\label{eq:est-D0}
    \sum_{k=2}^{K}\int_{(E^h_k\triangle E^h_{k-1})\cap \{|d_{E^h_{k-1}}|\ge hR\}}
 |z^h_k-\nabla d_{E^h_{k-1}}|\le 2 \sum_{k=2}^{K} \frac{D(E^h_k;E^h_{k-1})}{hR}\le\frac{2M}R\,,
\end{equation}
where in the last inequality we used \eqref{sums}. Inside the complement region $\{ |d_{E^h_{k-1}}|\le hR\}$, set 
\[
A^h_k:= (E^h_k\triangle E^h_{k-1})
\cap\{|d_{E^h_{k-1}}|<hR\}\,.
\]
Apply Cauchy--Schwarz 
on the space $L^2(\chi_{A^h_k}d\#(k)dx  )$, where $d\#$ is the counting measure, to get 
\begin{equation}\label{eq:est-D1}
    \begin{split}
 \sum_{k=2}^K\int_{A^h_k}
       |z^h_k-\nabla d_{E^h_{k-1}}|         \le
 \left(\sum_k|E^h_k\triangle E^h_{k-1}|\right)^{1/2}
 \left( \sum_{k}\int_{A^h_k}
       |z^h_k-\nabla d_{E^h_{k-1}}|^2\right)^{1/2}\\  
 \le
 \left(\sum_k|E^h_k\triangle E^h_{k-1}|\right)^{1/2}
 \left(C hR\sum_k  S^h_k+C h^{3/2}R^2
       \sum_k   P(E^h_{k-1})\right)^{1/2}\,,
\end{split}
\end{equation}
where in the second step we used \eqref{tubedefect} for the choice $r=hR$. 
By the summability of $|E^h_k\triangle E^h_{k-1}|$ by  \eqref{volsum} and the one of $S^h_k$ by \eqref{Ssum}. We deduce, upon combining \eqref{eq:est-D0} and \eqref{eq:est-D1} that
\[
 \sum_{k=2}^K\int_{E^h_k \triangle E^h_{k-1}}
       |z^h_k-\nabla d_{E^h_{k-1}}|          
 \le
  C_T
 \left( R^{-1} +h^{1/2}R^{1/2} +  h^{3/4}R
       \Big(\sum_k   P(E^h_{k-1})\Big)^{1/2}\right)\,,
\]
so that 
\eqref{asum} follows since $\sum_{k=2}^K   P(E^h_{k-1})\le KM\le TM/h$.

Finally apply \eqref{gaussd} with $b$ replaced by $\nabla\varphi$ and weight  $1$. The term containing $d_{E^h_{k-1}}\Delta\varphi$ is bounded by $C D(E^h_{k},E^h_{k-1})$; replacing $\nabla d_{E^h_{k-1}}$ by $z^h_k$ produces an error bounded by  $C\int_{{E^h_{k}}\triangle {E^h_{k-1}}} |   z^h_k-\nabla d_{E^h_{k-1}}|$. This shows \eqref{eq:transp-term-old}.
\end{proof}

\begin{lemma}[Auxiliary flow computation] 
Let $G=\Phi_h(F)$.
For $h>0$ sufficiently small, it holds
\begin{align}
& P_\varphi(G)-P_\varphi(F)
 =
 h\delta P(F)[\varphi b]
 +h\int_{\partial F}(b\cdot\nu_F)\,
       \partial_{\nu_F}\varphi\,d\mathcal H^n
 +\mathcal E\,, \label{perdef}
 \\
 &\frac1hD_\varphi(G,F)
 \le
 \frac h2\int_{\partial F}
 \varphi|b\cdot\nu_F|^2\,d\mathcal H^n
 +Ch^{3/2}P(F)\,,                                      \label{distdef}
 \\
 &\left|
 \int(\chi_G-\chi_F)z\cdot\nabla\varphi
 -h\int_{\partial F}(b\cdot\nu_F)
       \partial_{\nu_F}\varphi\,d\mathcal H^n
 \right|
 \le
 C\left(
 \int_{G\triangle F}|z-\nabla d_F|
 +h^{3/2}P(F)
 \right)\,,                                            \label{eq:transp-term-new}
 \\
 &\sum_{k=2}^K
 \int_{\Phi_h(E^h_{k-1})\triangle E^h_{k-1}}
 |z^h_k-\nabla d_{E^h_{k-1}}|
 \le C_T\bigl(h^{1/2}+h^{1/4}\bigr)\,,             \label{eq:sums-calerr}
\end{align}
where the error term $\mathcal E$ satisfies $|\mathcal E|
 \le Ch^2P(F)$.
\end{lemma}

\begin{proof}
Expand the  area formula 
$P_\varphi(\Phi_h(F))=\int_{\partial F}\varphi(\Phi_h(p))J_\tau\Phi_h(p)$.
Its first variation is
\[
\int_{\partial^* F}(\nabla\varphi\cdot b+\varphi\diver_\tau b)
=\delta P(F)[\varphi b]+\int_{\partial^*F}(b\cdot\nu_F)\partial_\nu\varphi\,;
\]
its second derivative is bounded by $CP(F)$: all the constants  appearing depend on  $\varphi, b$. This shows \eqref{perdef}.

Let $\pi_F$ denote the nearest-point projection in the tubular
neighborhood and define
\[
 \theta_h:=\pi_F\circ(\Phi_h\llcorner\partial F)\,.
\]
The curvature estimates \eqref{tubeest}   and  $\Phi_h(p)=p+hb(p)+O(h^2)$  
imply that $\theta_h$ is a diffeomorphism of $\partial F$ for
sufficiently small $h$ (independent of $F$). Consequently, $\partial G$ is a normal
graph over $\partial F$ induced by the height function $w_h$ defined via
\[
 w_h(\theta_h(p))=d_F(\Phi_h(p)), \quad \text{for all }p\in\partial F\,.
\]
Using \eqref{hessd} and \eqref{tubeest}, Taylor's formula gives
\[
 d_F(\Phi_h(p))
 =
 h\,b(p)\cdot\nu_F(p)+O(h^{3/2}), \quad  
 \|\nabla_\tau (b\cdot \nu_F)\|_\infty
 \le \frac{C}{\sqrt h}\,.
\]
Since $|\theta_h(p)-p|\le Ch$, it follows that
\[
 h\big|b(\theta_h(p))\cdot \nu_F(\theta_h(p))-b(p)\cdot \nu_F(p)\big|\le Ch^{3/2}\,.
\]
Therefore it holds 
\begin{equation}\label{heighterror}
 \|w_h-hb\cdot \nu \|_{L^\infty(\partial F)}\le Ch^{3/2}\,.
\end{equation}
Normal coordinates give
\[
\begin{split}
     D_\varphi(G,F)
 &=
 \int_{\partial F}\int_0^{ w_h(p)}
 \varphi(\Psi(p,s))\,sJ_\Psi(p,s)\,ds\,d\mathcal H^n(p)
 \\ 
 &=
 \frac{1}{2}
 \int_{\partial F}\varphi w_h^2\,d\mathcal H^n
 +O(h^{5/2}P(F))\\
 &=
 \frac{h^2}{2}
 \int_{\partial F}\varphi (b\cdot\nu_F)^2\,d\mathcal H^n
 +O(h^{5/2}P(F))\,,
\end{split}
\]
where we used  the estimates of Lemma \ref{lem:tubecalc} and  \eqref{heighterror}. This  concludes  \eqref{distdef}.

In order to show \eqref{eq:transp-term-new},  use again normal coordinates on $G$ 
$\nabla d_F(p+s\nu_F(p))=\nu_F(p)$ to deduce  
\begin{align*}
 \int(\chi_{  G}-\chi_F)
       \nabla d_F\cdot\nabla\varphi 
       &=
 \int_{\partial F}\int_0^{w_h(p)}
 \partial_{\nu_F}\varphi(\Psi(p,s))
 J_\Psi(p,s)\,ds\,d\mathcal H^n(p)\\
 &=
 h\int_{\partial F}
 (b\cdot\nu_F)\partial_{\nu_F}\varphi\,d\mathcal H^n
 +O(h^{3/2}P(F))\,,
\end{align*}
where we used again the estimates of Lemma \ref{lem:tubecalc} and \eqref{heighterror}.

Finally, set $A^h_k:=\Phi_h(E^h_{k-1})\triangle E^h_{k-1}\subseteq \{|d_{E^h_{k-1}}|<Ch\}$.  Use  Cauchy–Schwarz 
on the space $L^2( \chi_{A^h_k} d\#(k) dx )$  to deduce  
\[
\sum_{k=2}^K\int_{A^h_k}|z^h_k-\nabla d_{E^h_{k-1}}|
 \le 
 \left(Ch\sum_{k=2}^KP(E^h_{k-1})\right)^{1/2} 
 \left(
 \sum_{k=2}^K
 \int_{\{|d_{E^h_{k-1}}|<Ch\}}
 |z^h_k-\nabla d_{E^h_{k-1}}|^2
 \right)^{1/2}\,.
\]
Applying \eqref{tubedefect} with $r=Ch$ and using the summability of $P(E^h_k),S^h_k$   proves \eqref{eq:sums-calerr}.
\end{proof} 

\subsection{Brakke's inequality by duality }
We first derive an approximate Brakke inequality holding at the discrete level 
and then derive a limiting one for varifold limits of the discrete flows. The crucial point is testing the inequality with a fixed vector field,  which will be  chosen only after passing to the limit $h\to 0$.

\begin{proposition}[Discrete Brakke's inequality] 
Fix $b\in C^2_c(\R^{n+1};\R^{n+1})$, $ \varphi\in C^2_c(\R^{n+1};\R_+)$. Then, for $h$ small enough  and  for  every  $k\ge2$  
\begin{equation}\label{dualineq}
\begin{split}
 P_\varphi(E^h_k)-P_\varphi(E^h_{k-1})\le h\Bigl[
 &\delta P(E^h_k)[\varphi b]+\delta P(E^h_{k-1})[\varphi b]
 -\delta P(E^h_k)[\nabla\varphi]\\
 &+\int_{\partial E^h_{k-1}}\varphi|b\cdot\nu_{E^h_{k-1}}|^2d\HH^n\Bigr]+\rho^h_k\,,
\end{split}
\end{equation}
where $\rho^h_k\ge0$ can be chosen so that, for every $T>0$, it holds 
\begin{equation}\label{rho}
 \begin{split}
 \rho^h_k\le C\Bigl(&D(E^h_k,E^h_{k-1})+h^{3/2}P(E^h_{k-1})\\
 &+\int_{E^h_k\triangle E^h_{k-1}}|z^h_k-\nabla d_{E^h_{k-1}}|
 +\int_{\Phi_h(E^h_{k-1})\triangle E^h_{k-1}}
       |z^h_k-\nabla d_{E^h_{k-1}}|\Bigr)\,,\\
 \sum_{k=2}^{ [T/h]} \rho^h_k&\le C_T h^{1/8}\,.
 \end{split}
\end{equation}
The estimates are uniform in $[0,T]$ for  tests $\varphi,b$ with uniformly bounded  $C^2$  norms.
\end{proposition}

\begin{proof}
Fix $k\ge2$ and write for simplicity 
\[
 E=E^h_k,\qquad F=E^h_{k-1},\qquad z=z^h_k,
 \qquad G=\Phi_h(F)\,.
\]
Applying \eqref{local} with   the competitor $G$ above  and then adding
$P_\varphi(G)-P_\varphi(F)$, we obtain
\begin{equation}
 P_\varphi(E)-P_\varphi(F)
 \le
 -\frac{D_\varphi(E,F)}h
 +\frac{D_\varphi(G,F)}h                                      +\int(\chi_E-\chi_G)z\cdot\nabla\varphi
 +P_\varphi(G)-P_\varphi(F)\,.                         \label{dual-step-1}
\end{equation}
By \eqref{eq:transp-term-old},  \eqref{eq:transp-term-new}  and \eqref{perdef} respectively, we obtain 
\begin{align}
\int(\chi_E-\chi_F)z\cdot\nabla\varphi
 &\le
 -h\delta P(E)(\nabla\varphi)
 +C\left(
 D(E,F)+
 \int_{E\triangle F}|z-\nabla d_F|
 \right)\,,   \label{dual-old-flux} 
 \\
 -\int(\chi_G-\chi_F)z\cdot\nabla\varphi
 &\le 
 -h\int_{\partial F}
 (b\cdot\nu_F)\partial_{\nu_F}\varphi\,d\mathcal H^n +C\left(
 \int_{G\triangle F}|z-\nabla d_F|
 +h^{3/2}P(F)
 \right)\,,  \label{dual-new-flux}
 \\
 P_\varphi(G)-P_\varphi(F)
 &=
 h\delta P(F)[\varphi b]
 +h\int_{\partial F}
 (b\cdot\nu_F)\partial_{\nu_F}\varphi\,d\mathcal H^n
 +\mathcal E\,,                       \label{dual-perimeter} 
\end{align}  
where $
 |\mathcal E|\le Ch^2P(F)$. 
Adding \eqref{dual-new-flux} and \eqref{dual-perimeter}, the two terms involving $\partial_\nu \varphi$
 cancel. Therefore
\begin{equation}
 -\int(\chi_G-\chi_F)z\cdot\nabla\varphi
 +P_\varphi(G)-P_\varphi(F)                       \le
 h\delta P(F)[\varphi b]
 +C\left(
 \int_{G\triangle F}|z-\nabla d_F|
 +h^{3/2}P(F)
 \right)\,.                                          \label{dual-cancellation}
\end{equation}
On the other hand, the estimates  \eqref{endpointineq} and \eqref{distdef}   yields respectively 
\begin{align}
 -\frac{D_\varphi(E,F)}h
 &\le
 h\delta P(E)[\varphi b]
 +\frac h2\int_{\partial F}
 \varphi|b\cdot\nu_F|^2\,d\mathcal H^n
  +C\left(D(E,F)+h^{3/2}P(F)\right)\,,               \label{dual-endpoint}\\
 \frac{D_\varphi(G,F)}h
 &\le
 \frac h2\int_{\partial F}
 \varphi|b\cdot\nu_F|^2\,d\mathcal H^n
 +Ch^{3/2}P(F)\,.     \label{dual-competitor-distance}
\end{align}
Thus the two quadratic contributions in
\eqref{dual-endpoint} and \eqref{dual-competitor-distance}, each
having coefficient $1/2$, add to
\[
 h\int_{\partial F}\varphi|b\cdot\nu_F|^2\,d\mathcal H^n\,.
\]
Substituting \eqref{dual-old-flux},
\eqref{dual-cancellation}, \eqref{dual-endpoint} and
\eqref{dual-competitor-distance} into \eqref{dual-step-1} gives
\begin{equation*}
 P_\varphi(E)-P_\varphi(F)
 \le h\Big( \delta P(E)[\varphi b]+\delta P(F)[\varphi b]
 -\delta P(E)[\nabla\varphi]
 +\int_{\partial F}\varphi|b\cdot\nu_F|^2\,d\mathcal H^n
 \Big)+\rho^h_k\,,
\end{equation*}
where $\rho^h_k\ge0$ may be chosen so that
\begin{align*}
 \rho^h_k\le C\Bigg(
 &D(E^h_k,E^h_{k-1})+h^{3/2}P(E^h_{k-1}) +\int_{E^h_k\triangle E^h_{k-1}}
 |z^h_k-\nabla d_{E^h_{k-1}}|\\
 &+\int_{\Phi_h(E^h_{k-1})\triangle E^h_{k-1}}
 |z^h_k-\nabla d_{E^h_{k-1}}|
 \Bigg)\,.
\end{align*}
This proves  \eqref{dualineq}. 
It remains to prove the summability of $\rho^h_k$. By \eqref{sums} and $Kh\le T$,
\[
 \sum_{k=2}^K D(E^h_k,E^h_{k-1})\le hM,
 \qquad
 h^{3/2}\sum_{k=2}^K P(E^h_{k-1})
 \le TMh^{1/2}\,.
\]
Choose $R=h^{-1/8}$. For sufficiently small $h$, $  R \ge 1 $ and  $hR=h^{7/8}<a\sqrt h $. Therefore, by \eqref{asum} we get 
\[
 \sum_{k=2}^K
 \int_{E^h_k\triangle E^h_{k-1}}
 |z^h_k-\nabla d_{E^h_{k-1}}|
 \le
 C_T\left(h^{1/8}+h^{7/16}+h^{1/8}\right)
 \le C_T h^{1/8}\,.
\]
Finally, \eqref{eq:sums-calerr} gives
\[
 \sum_{k=2}^K
 \int_{\Phi_h(E^h_{k-1})\triangle E^h_{k-1}}
 |z^h_k-\nabla d_{E^h_{k-1}}|
 \le C_T(h^{1/2}+h^{1/4})\,.
\]
Combining these estimates and enlarging $C_T$, we conclude \eqref{rho}.
\end{proof}

Before showing Brakke's inequality, we introduce the following useful notation. Given a set of finite perimeter $E$ we denote the multiplicity one varifold associated with its reduced boundary as 
\[
V_E := |D\chi_E|\otimes \delta_{T_x\partial^* E}\,.
\]
We also use the standard notation $\|V\|$ to denote the weight measure associated to the varifold $V$. Instead, $\llbracket E \rrbracket$ will denote the multiplicity one current 
\[
\llbracket E \rrbracket = \llbracket E, e_1 \wedge\ldots\wedge e_{n+1},1 \rrbracket\,,
\]
so that 
\[
\partial\llbracket E \rrbracket = \llbracket \partial^* E, \star\nu_E , 1 \rrbracket \quad \text{and}\quad \| \partial \llbracket E \rrbracket \| = \|V_E\|=|D\chi_E|\,.
\]
Finally, following White \cite{Whi09} we denote by $[V]$ the mod(2) flat chain associated with the integral varifold $V$.

The first step of the proof is rather standard and relies on the uniform $L^2$- bound on the mean curvature holding for the discrete evolution \eqref{sums}. We refer for instance to \cite[Section 3]{Ton}.

\begin{theorem}
\label{thm:brakke}
Assume $1\le n\le6$. Let  $E_0\subset\mathbb R^{n+1}$ be a bounded set 
with finite perimeter with $|\partial E_0|=0$, and let
$E_t$ be a flat flow obtained as $L^1$ limit of $E^{h_j}_t$ as $h_j\to 0$.

Then, after passing to a further subsequence, there exists a measurable
family of  $n$-varifolds $\{V_t\}_{t\ge0}$ which are integral for a.e. $t\ge 0$ and   such that
\[
\begin{split}
  V_{E_t^{h_j}}dt
 &\rightharpoonup V_tdt\quad \text{as varifolds in space-time}\,, \\
 \|V_{E_t^{h_j}}\| &\rightharpoonup \|V_t\|  \quad \text{for every $t \geq 0$, as measures in space}\,,\\
 \{V_t\}_{t \geq 0} &\mbox{ is a Brakke flow with } V_0=V_{E_0}\,.
 \end{split}
\]
Moreover, we have:
\begin{itemize}
\item[(i)] \textit{Continuity at $t=0$}: as $t \to 0^+$
\begin{equation} \label{init-cont}
\begin{split}
    P(E_t) &\to P(E_0)\,,\\ 
    \|V_t\|(\R^{n+1}) &\to P(E_0)\,, \\
    \|V_t\| &\rightharpoonup |D\chi_{E_0}| \mbox{ as Radon measures}\,,\\
\|\|V_t\| - |D\chi_{E_t}| \|_{\mathrm{TV}} &\to 0\,.
    \end{split}
\end{equation}
    \item[(ii)]  \textit{Domination}: 
    \begin{equation}\label{domin}
    \|V_{E_t}\|=|D\chi_{E_t}|\le  \|V_t\|
 \quad\text{for every }t>0\,,
 \end{equation}
\item[(iii)] \textit{$\mathrm{mod}(2)$ compatibility}:
\begin{equation} \label{mod2-equiv}
    [V_t] = \partial\llbracket E_t\rrbracket \quad \mathrm{mod}(2) \qquad \text{for a.e. $ t > 0$}\,.
\end{equation}
\end{itemize}
Finally, the dissipation inequality holds:
\begin{equation} \label{eq:dissipation}
    \|V_t\|(\R^{n+1}) + \int_0^t \int |\mathbf H_\tau|^2 \, d\|V_\tau\|\,d\tau \leq  P(E_0) \qquad \text{for every $t \geq 0$}\,.
\end{equation}
\end{theorem}

\begin{remark}
    We observe explicitly that \eqref{mod2-equiv} is equivalent to the fact that for a.e. $t > 0$ there exists an integral varifold $Z_t$ so that 
    \begin{equation} \label{mod2-equiv-2}
        V_t = V_{E_t} + 2Z_t\,.
    \end{equation}
    In particular, setting $V_t = \mathbf{var}(M_t,\theta_t)$, we have that
    \begin{equation} \label{eq:parity}
        \theta_t (x) \begin{cases}
            \mbox{is an odd integer} & \mathcal H^n\mbox{-a.e. } x \in \partial^*E_t  \\
            \mbox{is an even integer} & \mathcal H^n\mbox{-a.e. } x \in M_t \setminus \partial^*E_t \,. \\
        \end{cases}
    \end{equation}
\end{remark}

\begin{proof}
Fix a sequence $h_j\to0$ realizing the flat flow. Throughout the proof, in
order to lighten the notation, we suppress the index $j$ and write simply $h$
in place of $h_j$; in particular, every limit as $h\to0$ is understood along
this sequence. We write
\[
 \mu_t^h:=\|V_{E_t^h}\|=|D\chi_{E^h_t}|,
 \qquad
 V_t^h:=V_{E_t^h},
 \qquad
 M:=P(E_0)\,.
\]
For notational  simplicity extend $E_{t}^{h}\equiv E_0$ for $t<0$ and the definitions of $\mu^h_t,V^h_t$ accordingly. We also set $\mathbf H_t^h := \mathbf H_{E^h_t}$ for $t \geq h$, and, in order to enforce the initial condition, we define $\mathbf H^h_t:=0$ for $t\in [0,h)$.
Given a time-dependent map $f(x,t)$ we use the short-hand notation $f_t$ to denote $f(\cdot,t)$.

\smallskip

\noindent\textbf{Step 1:} \textit{Varifold limit and integrality}. 
For $k\ge 1$, minimality and the basic estimates \eqref{sums} give
\[
 P(E^{h}_{kh})+\frac1hD(E^{h}_{kh},E^{h}_{(k-1)h})\le P(E^{h}_{(k-1)h}),
 \qquad
 \sup_t\mu_t^h(\mathbb R^{n+1})\le M\,.
\]
Monotonicity of the perimeters along discrete flows is thus obvious. We show an estimate on the weighted perimeter $P_\varphi(E^h_{kh})$. 
For every nonnegative spatial test function  $\varphi\in C_c^2(\R^{n+1};\R_+)$, by \eqref{local} with $G=E^h_{(k-1)h}$ we get 
\[
P_\varphi(E^{h}_{kh})-P_\varphi(E^{h}_{(k-1)h}) \le -\frac{D_\varphi(E^{h}_{kh},E^{h}_{(k-1)h})}{h} + \int (\chi_{E^{h}_{kh}}-\chi_{E^{h}_{(k-1)h}}) z\cdot \nabla\varphi\,.
\]
Since $D_\varphi\ge 0$, we deduce  
\[
\left( P_\varphi(E^{h}_{kh})-P_\varphi(E^{h}_{(k-1)h}) \right)_ +\le  \|\nabla\varphi\|_\infty |E^{h}_{kh}\triangle E^{h}_{(k-1)h}|\,.
\]
Summing this estimate and using \eqref{volsum} we deduce 
\[
\sum_{k\le [T/h]} ( P_\varphi(E^{h}_{kh})-P_\varphi(E^{h}_{(k-1)h}))_+ \le \|\nabla\varphi\|_\infty\left( C(T+1)M  + |E^h_h\triangle E_0| \right) \,.
\]
To bound the negative increments note that 
\[
P_\varphi(E^{h}_T)-P_\varphi(E_0)= \sum_{k\le [T/h]} ( P_\varphi(E^{h}_{kh})-P_\varphi(E^{h}_{(k-1)h}))_+ - \sum_{k\le [T/h]} ( P_\varphi(E^{h}_{kh})-P_\varphi(E^{h}_{(k-1)h}))_-\,,
\]
and since 
$0\le P_\varphi(E^{h}_{kh})\le M\|\varphi\|_\infty$, also  the sum of the negative  increments of $P_\varphi$ is uniformly bounded. In particular, $t\mapsto P_\varphi(E^{h}_t)$ is uniformly bounded in $BV([0,T])$  for every fixed $T>0$. The same estimate holds for every non-positive test function $\varphi\le 0$.
For a signed test function $\psi\in C^1_c$, we can consider $  \eta\in C^1_c(\R^{n+1};\R_+)$ and $\eta\ge 1$ on $\spt(\psi)$ so that 
\[
\psi=(\psi+C\eta)-C\eta\,,
\]
where $C>\|\psi\|_\infty$, and then apply the previous estimate to the two signed functions on the \textit{rhs} above.
Choose then a countable family $\mathcal D$ of functions in $C^1_c(\R^{n+1})$ that are dense in $C^0_c(\R^{n+1})$. 
Helly's selection theorem,
applied to every element of $\mathcal D$ and every interval $[0,\ell]$ for $\ell\in\N$, given one subsequence such that  $P_\psi(E^{h}_t)\to L_t(\psi)$ for every $t\ge 0$ and $\psi\in \mathcal D$.
Since $  P_f(E^h_t)\le M\|f\|_{\infty} $ holds for every  $f \in C^{0}$, by density  $L_t$ extends uniquely to a positive bounded functional on $C^0_c(\R^{n+1})$.
By Riesz representation theorem,  it  is represented by a Radon measure $\mu_t$. We conclude that 
\begin{equation}\label{eq:conv-weight1}
    \mu^h_t\rightharpoonup \mu_t,\quad \text{for every }t\ge 0\,. 
\end{equation}
Note that at $t=0$ it holds $\mu_0=|D\chi_{E_0}|$ by the very definition of the discrete scheme.

Along the subsequence selected above,   measure compactness in space-time gives a common
subsequence such that
 \[
 V_t^hdt \rightharpoonup V_t dt\,.
\]
Space-time convergence of $\|V^h_t\|dt=\mu^h_tdt\rightharpoonup \|V_t\|dt $  and dominated convergence shows that 
\begin{equation}\label{eq:conv-weight2}
     \|V_t\|=\mu_t
 \quad\text{for a.e. }t\,.
\end{equation}
Indeed, for every $f\in C^0_c((0,T)),\varphi\in C^0_c(\R^{n+1})$ it holds 
\[
\int f(t) \varphi(x) d\|V_t\|(x)dt = \lim_{h\to 0}
\int f \varphi\, d\|V^h_t\|dt = \lim_{h\to 0}
\int f\int \varphi \,d\mu^h_tdt =
\int f\int \varphi \,d\mu_t dt   \,.
\]
We need to show that the varifolds $V_t$ are integral with curvatures  in $L^2$, for a.e. $t\ge 0$.
Set
\[
 e_h(t):=\int|\mathbf H_t^{h}|^2\,d\mu_t^h,\qquad t\ge 0\,.
\]
By \eqref{sums}, $e_h\in L^1_{loc}([0,+\infty))$ and $\int_0^T e_h(t)\,dt\le CM$. 
For every $b\in C^1_c(\R^{n+1}\times[0,T);\R^{n+1})$  it holds 
\begin{equation}\label{eq:fvar-curv}
     \int_0^T\delta V_t^h(b)\,dt
 =
 -\int_0^T\int\mathbf H_t^{h}\cdot b\,d\mu_t^h\,dt + \int_0^h\delta V_t^h(b)\,dt\,,
\end{equation}
where we recall that $\mathbf H^h_\tau:=0$  for $\tau\in [0,h)$. The last integral in \eqref{eq:fvar-curv}  vanishes as $h\to0$. Indeed, by the first variation formula, for every $\tau\in [0,h)$ it holds $ \delta V_\tau ^h(b(\cdot,\tau))=\int_{\partial^* E_0} \diver_\tau b(\cdot,\tau)$. This term is uniformly bounded, thus integrating and sending $h\to 0$ yields the previous claim.

Since $  V^h_tdt \rightharpoonup   V_tdt $ we can pass to the limit on the \textit{lhs} of \eqref{eq:fvar-curv}. Applying Cauchy--Schwarz
on the right gives
\[
 \left|\int_0^T\delta V_t(b)\,dt\right|
 \le
 \left(\liminf_{h\to0}\int_0^T e_h(t)\,dt\right)^{1/2}
 \left(\int_0^T\int|b|^2\,d\mu_t\,dt\right)^{1/2}\,.
\]
Riesz representation  and disintegration therefore give $\mathbf H\in L^2(d\mu_t dt)$ such that 
\[
 \delta V_t=-\mathbf H_t\mu_t
 \quad\text{for a.e. }t,
 \qquad
 \int_0^T\int|\mathbf H_t|^2\,d\mu_t\,dt\le CM\,.
\]
Therefore, for a.e.  $t$, the varifold  $V_t$ has locally bounded
first variation. 

To  show that $V_t$ is integral for a.e. $t$, we first note that by Fatou's lemma $\liminf_{h\to0} e_h(t)<\infty$ for
a.e. $t$. For each such $t$, choose a subsequence (depending on $t$)
with bounded $e_h(t)$. Since 
\begin{equation}\label{eq:hp-Allard}
    \|\delta V_t^h\|(\mathbb R^{n+1})\le  M^{1/2}e_h(t)^{1/2}\,,
\end{equation}
which is bounded along the subsequence,  
Allard's integral compactness theorem \cite{Allard} provides a further subsequence and an integral varifolds  $W_t$ with $V^h_t\rightharpoonup W_t$. By the convergence \eqref{eq:conv-weight2} it holds $\|W_t\|=\|V_t\|$ for a.e. $t$.  Since  the varifold $V_t$ has bounded first variation  and its weight $\|W_t\|$ has positive
integer density almost everywhere, Allard's rectifiability
theorem implies that $V_t$ is rectifiable.
Since a rectifiable varifold is determined by its weight,
\begin{equation}\label{eq:V=W}
 V_t=W_t
 \quad\text{for a.e.  }t\,.   
\end{equation}
Thus the spacetime limit itself is integral for a.e. $t$  and has no singular
first variation. Its mean curvature $\mathbf H_t$ is normal
to its approximate tangent plane by Brakke's perpendicularity theorem \cite[Theorem 5.8]{Bra}.

\smallskip

\noindent\textbf{Step 2:} \textit{$\mathrm{mod}(2)$ compatibility}.
Consider the full measure set of times  $t$ as in Step 1 such that \eqref{eq:hp-Allard} holds. Let $T_t^h := \llbracket E_t^h \rrbracket$, so that, as observed above, $\|V_t^h\| = \| \partial T_t^h \|$. Notice also that $[V_t^h] = \partial T_t^h \quad \mathrm{mod}(2)$, namely $V_t^h$ and $\partial T_t^h$ are compatible, in the language of \cite{Whi09}. Let also $T_t := \llbracket E_t \rrbracket$, and let $\mathcal F(T_1,T_2)$ denote the integral flat distance between flat chains $T_1, T_2$. Since
\[
\mathcal F(T_t^h, T_t) \leq |E_t^h \triangle E_t|\,,
\]
as a consequence of \eqref{flat-convergence} we have that along a suitable subsequence $h_j$ (depending on $t$) we have, as $j \to \infty$ that both hold:
\[
\begin{split}
   & V_t^{h_j} \mbox{ satisfy the assumptions of Allard's compactness, and } V_t^{h_j} \rightharpoonup V_t\,, \\
    & \mathcal F(\partial T_t^{h_j}, \partial T_t) \to 0 \,.
\end{split}
\]
By \cite[Theorem 1.2]{Whi09}, $V_t$ and $\partial T_t$ are compatible, that is \eqref{mod2-equiv} holds. 

\smallskip

\noindent\textbf{Step 3:} \textit{Brakke's inequality}.
Fix tests in space-time $  \varphi\in C_c^2(\mathbb R^{n+1}\times[0,\infty);\R_+)$ and  $b\in C_c^2(\mathbb R^{n+1}\times[0,\infty);
          \mathbb R^{n+1})$. 
Let $0<s<t$ and take  $h$ sufficiently small so that  $2\le [s/h]<[t/h]$.  
The convergence established in Step~1 and 
  the 
uniform mass bound imply 
\[
\mu_{s}^h(\varphi_{[s/h]h})\to \mu_s(\varphi_s), \quad \mu_{t}^h(\varphi_{[t/h]h})\to \mu_t(\varphi_t)\,.
\]
Sum \eqref{dualineq} for $k=[s/h]+1,\ldots, [t/h] $, using
the spatial tests $\varphi_{kh}$ and $b_{kh}$, and add the
temporal increments
\[
P_{\varphi_{kh}}(E^h_{k-1})
-
P_{\varphi_{(k-1)h}}(E^h_{k-1})\,.
\]
The left-hand side is a telescopic sum, equal to 
\[
\mu_{[t/h]h}^h(\varphi_{[t/h]h})
-
\mu_{[s/h]h}^h(\varphi_{[s/h]h})=\mu_{t}^h(\varphi_{[t/h]h})
-
\mu_{s}^h(\varphi_{[s/h]h})\,,
\]
while the temporal increments satisfy 
\[ 
\sum_{k=[s/h]+1}^{[t/h]}
\left[
P_{\varphi_{kh}}(E^h_{k-1})
-
P_{\varphi_{(k-1)h}}(E^h_{k-1})
\right]
=
\int_{[s/h]h}^{[t/h]h}
\int \partial_\tau\varphi\,d\mu_\tau^h\,d\tau
\underset{h\to 0}{\longrightarrow}
\int_s^t
\int \partial_\tau\varphi\,d\mu_\tau\,d\tau\,,
\]
where $\partial_\tau$ denotes derivative in the time variable. 
For fixed smooth tests, the  time integrals of the expressions in  the \textit{rhs} of 
\eqref{dualineq} is continuous with respect to  varifold convergence in space-time. Note also that the   uniform mass bound implies  
\[
\lim_{h\to 0} \int_s^t \delta V^h_\tau(\varphi_\tau b_\tau) \,d\tau=\lim_{h\to 0}  \int_s^t \delta V^h_{\tau+h}(\varphi_\tau b_\tau)\,d\tau  =  \int_s^t\delta V_\tau(\varphi_\tau b_\tau)\,d\tau\,.
\]
Using again the   uniform mass bound and the error estimate \eqref{rho}, we can pass to the limit for every $0<s<t$
\begin{equation}\label{eq:dual-limit}
\mu_t(\varphi_t)-\mu_s(\varphi_s)
\le \int_s^t \Bigg[
\int \partial_\tau\varphi\,d\mu_\tau
+2\delta V_\tau(\varphi_\tau b_\tau)
-\delta V_\tau(\nabla\varphi_\tau)
+\int \varphi_\tau
  |P_{T_xV_\tau}^{\perp}b_\tau|^2\,d\mu_\tau
\Bigg]\,d\tau\,.
\end{equation}
The equation \eqref{eq:dual-limit} also holds for $s=0$. Indeed, apply \eqref{local} at $k=1$ with $G=E_0$ and spatial test $\varphi_h$ to get 
\[
\mu_h^h(\varphi_h)-\mu_0(\varphi_0)
\le
\|\nabla\varphi_h\|_\infty
|E_h^h\triangle E_0| + M \|\varphi_h-\varphi_0\|_\infty\,.
\]
On the other hand, $|E_h^h\triangle E_0| = |E^h_1 \triangle E_0|\to0$ as $h\to 0$ as a consequence of the assumption $|\partial E_0|=0$ by the argument in Remark \ref{rmk:c0}, and, of course, $\|\varphi_h-\varphi_0\|_\infty \to 0$: hence, this time increment vanishes and the discussion above can be repeated to include $s=0$. 

By Step~1, $\delta V_\tau=-\mathbf H_\tau\mu_\tau$ and
$\mathbf H_\tau$ is normal to $T_xV_\tau$ for almost every
$\tau$. Hence the right-hand side of \eqref{eq:dual-limit}
equals
\[
\int_s^t\int
\left(
\partial_\tau\varphi
+\nabla\varphi\cdot\mathbf H_\tau
-\varphi|\mathbf H_\tau|^2
+\varphi
 |P_{T_xV_\tau}^{\perp}b_\tau-\mathbf H_\tau|^2
\right)\,d\mu_\tau\,d\tau\,.
\]
Smooth compactly supported vector fields are dense in $L^2(\varphi\mu_\tau d\tau;\R^{n+1})\ni \mathbf{H}$. Hence, taking the infimum over all $b$ of the previous display proves, for every $0 \leq s < t$
\[
\mu_t(\varphi_t)-\mu_s(\varphi_s)
\le \int_s^t\int
\left(
\partial_\tau\varphi
+\nabla\varphi\cdot\mathbf H_\tau
-\varphi|\mathbf H_\tau|^2
\right)\,d\mu_\tau\,d\tau\,.
\]
By Step 1 then $\|V_t\|=\mu_t$ for a.e. $t\ge 0$. 
Defining $V_t$ on the remaining exceptional times so that
$\|V_t\|=\mu_t$ (for instance, setting $V_0=V_{E_0}$ and $V_t := \mu_t \otimes \delta_{S}$ for an arbitrary $n$-plane $S$ at any exceptional bad time $t > 0$) completes the proof of Brakke's inequality for nonnegative tests $\varphi$ of class $C^2$. For a nonnegative $C_c^1$ test, choose nonnegative smooth approximations converging in $C^1$, with supports in a common compact set. The mass bound, the $L^2(d\mu_t\,dt)$ bound for
$\mathbf H$, and Cauchy--Schwarz allow passage to the limit in every term, so that the asserted $C_c^1$ formulation follows. 

\smallskip

\noindent \textbf{Step 4}: \textit{Dissipation, domination, and continuity at the initial time}

The dissipation inequality \eqref{eq:dissipation} is now an immediate consequence of Brakke's inequality after testing with $\varphi \equiv 1$ (and a suitable cut-off outside of the large ball containing the evolution). The initial datum satisfies 
$\mu_0=|D\chi_{E_0}|$, so that $\|V_0\|(\R^{n+1}) = \mu_0(\R^{n+1})=P(E_0)$.

For every $t\ge0$, the convergence
$E_t^h\to E_t$ in $L^1$, lower semicontinuity of perimeter,
and $\mu_t^h\rightharpoonup\mu_t$ imply
\[
|D\chi_{E_t}|(\psi)
\le
\liminf_{h\to0}|D\chi_{E_t^h}|(\psi)
=
\mu_t(\psi)
\qquad
\text{for every }0\le\psi\in C_c^0(\mathbb R^{n+1})\,.
\]
Thus $|D\chi_{E_t}|\le\mu_t$ for every $t\ge0$, thus proving domination.

As for continuity at $t=0$, by $L^1$ continuity of $E_t$, lower semicontinuity of perimeter, domination, and the global mass bound,
\[
    P(E_0)\le\liminf_{t\downarrow0}P(E_t)
     \le\liminf_{t\downarrow0}\|V_t\|(\mathbb R^{n+1})\le P(E_0)\,.
\]
Since both $P(E_t)$ and $\|V_t\|(\R^{n+1})$ are bounded above by $P(E_0)$, they both converge to $P(E_0)$.
For any sequence $t_j\downarrow0$, the uniform confinement and mass bound
permit extraction of a weak limit $\|V_{t_j}\|\rightharpoonup\nu$.
For every nonnegative $\psi\in C_c^0$, again by lower semicontinuity of perimeter and domination
\[
    |D\chi_{E_0}|(\psi)
    \le\liminf_j|D\chi_{E_{t_j}}|(\psi)
    \le\nu(\psi).
\]
The mass bound forces $\nu=|D\chi_{E_0}|$, proving weak continuity. Finally, by domination the measure $\|V_t\|-|D\chi_{E_t}|$ is nonnegative, and thus its total variation is simply $\|V_t\|(\R^{n+1}) - P(E_t)$, which tends to zero. This completes the proof.
\end{proof}

\section{Generalized BV flows and consequences}\label{sec:main}

We recall the notions of $L^2$-flow and generalized BV flow that will be
used below. We refer to \cite{MugRog,StuTon,Tas26}
for further details.

\begin{definition}[$L^2$-flow]
Let $\{V_t\}_{t\in [0,T)}$, possibly with $T=+\infty$, be a (measurable) family of $n$-varifolds in
$\mathbb{R}^{n+1}$, and set $d\mu(x,t) := d\|V_t\|(x)\,dt.$ 
We say that $\{V_t\}_{t\in[0,T)}$ is an \emph{$L^2$-flow} if the following hold:
\begin{itemize}
    \item[i)] for a.e.
$t\in(0,T)$, $V_t$ is integral, with bounded first variation absolutely continuous with respect to $\|V_t\|$ and generalized mean curvature vector $\mathbf H(\cdot,t)$; 
    \item[ii)] there exists a vector field $v\in L^2_{\mathrm{loc}}
    (\mu;\mathbb{R}^{n+1})$,  such that $ v(x,t)\perp T_xV_t$ for  $\mu$-a.e. $(x,t)$;
    \item[iii)]  for every compact set
$K\subseteq \mathbb{R}^{n+1}\times(0,T)$, there exists $C_K<\infty$ such that
\[
    \left|
    \int_{\mathbb{R}^{n+1}\times(0,T)}
    \bigl(\partial_t\phi+\nabla\phi\cdot v\bigr)\,d\mu
    \right|
    \leq C_K\|\phi\|_{L^\infty}\,,
\]
for every
$\phi\in C^1(\mathbb{R}^{n+1}\times(0,T))$
with $\operatorname{spt}\phi\subset K$.
\end{itemize}
The vector field $v$ defined above is called the velocity of the $L^2$-flow.
\end{definition}
    
The following notion of generalized BV flow has been introduced in \cite{StuTon}.

\begin{definition}[Generalized BV flow]
A pair $(\{V_t\}_{t\in[0,T)},\{E_t\}_{t\in[0,T)})$ 
is called a generalized BV flow in an open set $U \subset \R^{n+1}$ if $\{V_t\}$ is a Brakke flow in $U$, each $E_t$ is of (locally) finite perimeter in $U$,
\[
    |D\chi_{E_t}|\leq \|V_t\|
    \qquad\text{for every }t,
\]
and, denoting by $\mathbf H_t$ the generalized mean curvature of $V_t$,
\[
    \int_{E_{t_2}}\phi(x,t_2)\,dx
    -
    \int_{E_{t_1}}\phi(x,t_1)\,dx =
    \int_{t_1}^{t_2}\int_{E_t}
    \partial_t\phi\,dx\,dt
    +
    \int_{t_1}^{t_2}
    \int_{\partial^*E_t}
    \phi\,\mathbf H_t\cdot\nu_{E_t}
    \,d\mathcal H^n\,dt
\]
for every  $\phi\in C^1_c(U \times [0,T))$ and for every $0 \leq t_1 < t_2 < T$.
\end{definition}

The inequality
\[
    |D\chi_{E_t}|\leq\|V_t\|
\]
allows the  varifold to retain additional multiplicity possibly lost at the level of the sets $E_t$ due to cancellations phenomena of oriented boundaries. 

The notions above are linked by the following result. It is implicitly proved in \cite{StuTon,Tas25} and then proved in the present generality in \cite{Tas26}. We present the statement of the result for the two-phases problem, compare \cite[Theorem 1.2]{Tas26}.

\begin{theorem}
\label{thm:L2-to-BV}
Let $\{V_t\}_{t\in[0,T)}$ be an  $L^2$-flow in
$\mathbb{R}^{n+1}$ with space-time measure $d\mu=d\|V_t\|dt$ and velocity $v$, and let $\{E_t\}_{t\in [0,T)}$ be a   family of sets of finite perimeter satisfying
\[
    |D\chi_{E_t}|
    \leq \|V_t\|, \qquad \text{for every }t\in [0,T)\,.
\]
Define the corresponding space-time tube $\mathcal E:= \{(x,t)\in\mathbb{R}^{n+1}\times[0,T):
      x\in E_t\}$.
Assume that
\begin{itemize}
    \item[i)] it holds  $\Theta^{*n}(\|V_t\|,x)
    :=
    \limsup_{r\searrow0}
    \frac{\|V_t\|(B_r(x))}{\omega_nr^n}
    <\infty$ for a.e. $t$ and $\|V_t\|$-a.e. $x$;
    \item[ii)] it holds $
    |D_{x,t}\chi_{\mathcal E}|\ll \|V_t\|dt
    $ in $
    \mathbb{R}^{n+1}\times(0,T)$;
    \item[iii)] the tube $\mathcal E$ satisfies  $
    \chi_{\mathcal E}    \in
    C_{\mathrm{loc}}^{1/2}
    \bigl([0,T);L^1(\mathbb{R}^{n+1})\bigr)
    \cap
    BV_{\mathrm{loc}}
    \bigl(\mathbb{R}^{n+1}\times[0,T)\bigr)$.
\end{itemize}
Then, the velocity $v$ belongs to $L^2_{loc}(|D\chi_{E_t} | dt ;\R^{n+1})$, and for every 
$0\leq t_1<t_2<T$ and every
$\phi\in C_c^1(\mathbb{R}^{n+1}\times[0,T))$ it holds 
\begin{equation*}
    \int_{E_{t_2}}\phi(x,t_2)\,dx
    -
    \int_{E_{t_1}}\phi(x,t_1)\,dx =
    \int_{t_1}^{t_2}
    \int_{E_t}
    \partial_t\phi\,dx\,dt
    +
    \int_{t_1}^{t_2}
    \int_{\partial^*E_t}
    \phi\,v\cdot\nu_{E_t}
    \,d\mathcal H^n\,dt
\end{equation*}
\end{theorem}

Using the previous result we can show that flat flows are generalized BV.

\begin{theorem}\label{thm:main}
    Assume $1\le n\le 6$. Let $E_0\subseteq \R^{n+1} $ be a bounded set with finite perimeter and $|\partial E_0|=0$, and consider a flat flow $E_t $ starting from $E_0$. Let $V_t$ be the integral Brakke flow of Theorem \ref{thm:brakke}. Then, $\left( \{V_t\}_{t \geq 0}, \{E_t\}_{t \geq 0} \right)$ is a generalized BV flow: for every $0\leq t_1<t_2<+\infty$ and every
    $\phi\in C_c^1(\mathbb{R}^{n+1}\times[0,\infty))$ it holds 
    \begin{equation}
    \int_{E_{t_2}}\phi(x,t_2)\,dx
    -
    \int_{E_{t_1}}\phi(x,t_1)\,dx =
    \int_{t_1}^{t_2}
    \int_{E_t}
    \partial_t\phi\,dx\,dt
    +
    \int_{t_1}^{t_2}
    \int_{\partial^*E_t}
    \phi\,\mathbf H_{t}\cdot\nu_{E_t}
    \,d\mathcal H^n\,dt
    \label{eq:BV-gener}.
\end{equation}
\end{theorem}

\begin{proof}
    The proof  essentially follows from Theorem \ref{thm:brakke} and the arguments of \cite{StuTon}.
    Note first that $V_t$ is a Brakke flow, thus  by the proof of \cite[Thm 4.3]{StuTon} it is also an $L^2$-flow with velocity $v(\cdot,t)=\mathbf H_{t}$.  By the integrality of $V_t$, item \textit{i)} in Theorem \ref{thm:L2-to-BV} follows. 
   
    We then show item \textit{ii)}. Set $M=P(E_0)<+\infty$, and recall the notation $\mathcal E$ for the space-time tube of the flat flow. Note first that $D_x\chi_{\mathcal E}=D\chi_{E_t}\,dt$. Indeed, for every $b\in C_c^1(\mathbb R^{n+1}\times(0,T);\mathbb R^{n+1})$,  
    Fubini's theorem  and Gauss-Green (which can be applied since $E_t$ is of finite perimeter for every $t$) give
    \[
     \langle D_x\chi_{\mathcal E}, b\rangle
     =-\iint_{\mathcal E}\diver_x b\,dx\,dt
     =\iint b(\cdot,t)\cdot dD\chi_{E_t}\,dt\,.
    \]
   Then, by \eqref{domin} it holds 
\[
|D_x \chi_{\mathcal E}|=|D\chi_{E_t}| dt =\|V_{E_t}\|dt \le \|V_t\|dt,
\]
thus we only need to show the absolute continuity of $D_t\chi_{\mathcal E}$ with respect to $\|V_t\|dt$. 
We follow the argument of
\cite[Prop~4.5 and Cor~4.6]{StuTon}.
The result can be shown by using the estimates  of
\cite[Proposition 4.0.1]{MSS} \footnote{We explicitly note that this result does not require the convergence of the perimeter energies: in \cite{MSS} the convergence of the perimeters is only needed to characterize the limiting velocity as the mean curvature of the evolving sets.}.
In order to make the proof self-contained, we present a different proof based on the localized estimates of Lemma \ref{lem:tubecalc}. 

Fix a test function $\varphi\in C^2_c(\R^{n+1} \times (0,T))$, and set $\varphi_{kh}(x) :=\varphi(x,kh)$. For $k \geq 2$, also denote $\sigma=\operatorname{sgn}(d_{E^h_{k-1}})$.
Then
\[ 
\begin{aligned} \int(\chi_{E^h_k}-\chi_{E^h_{k-1}})\varphi_{kh} -\frac{D({E^h_{k}},{E^h_{k-1}})}h &= \int|\chi_{E^h_{k}}-\chi_{E^h_{k-1}}|\left(\sigma \varphi_{kh}-\frac{|d_{E^h_{k-1}}|}{h}\right)\\ &\le \int\left(\sigma \varphi_{kh}-\frac{|d_{E^h_{k-1}}|}{h}\right)_+. 
\end{aligned} 
\]
The positive part is supported in $|d_{E^h_{k-1}}|\le h\|\varphi\|_\infty$. For sufficiently small $h$ (depending on $\|\varphi\|_\infty$), it is then supported in the $O(\sqrt{h})$-neighborhood of $E^h_{k-1}$ where we can use normal coordinates by Lemma \ref{lem:tubecalc}. The same computations from Lemma \ref{lem:diss-bd} yield
\begin{equation}\label{eq:incr-dh}
    \int(\chi_{E^h_{k}}-\chi_{E^h_{k-1}})\varphi_{kh} \le \frac{D({E^h_{k}},{E^h_{k-1}})}h +\frac h2\int_{\partial {E^h_{k-1}}}\varphi_{kh}^2\,d\mathcal H^n +C_{\varphi} h^{3/2}M\,,
\end{equation}
for all $k \geq 2$. For $k=1$, the left-hand side vanishes for $h$ sufficiently small, because $\varphi$ vanishes in a neighborhood of $t=0$. Now, sum the discrete increments \eqref{eq:incr-dh} with $2 \leq k \leq [T/h]$ and then pass to the limit $h\to 0$. By discrete integration by parts, and using that $\varphi$ vanishes both near $t=0$ and $t=T$, we have 
\[
\sum_k \int (\chi_{E_k^h}-\chi_{E^h_{k-1}}) \varphi_{kh} = - \int_0^T \int_{E^h_t} \partial_t \varphi\,dxdt \xrightarrow{h \to 0} -\int_0^T\int_{E_t} \partial_t \varphi\, dxdt\,.
\]
The quadratic term converges to $\frac12 \int \varphi^2\,d\mu_t\,dt$ by measure convergence and uniform continuity of the test, whereas the sum of the remainders tends to zero. By  the summability of the dissipations \eqref{sums} we then get  
\[ 
    D_t\chi_{ \mathcal E}(\varphi) \le M+\frac12\int \varphi^2\,d\mu_tdt. 
\]
Apply this inequality replacing $\varphi$ with $\lambda \varphi$ and $-\lambda \varphi$, for each
fixed $\lambda>0$. Optimizing \emph{after} the limit $h\to0$ yields
\[
    | D_t\chi_{\mathcal E}(\varphi)|
    \le\inf_{\lambda>0}
       \left(\frac M\lambda+\frac\lambda2\|\varphi\|_{L^2(\mu)}^2\right)
    =\sqrt{2M}\,\|\varphi\|_{L^2(\mu)}.
\]
Therefore,
\[
D_t\chi_{\mathcal E}=w\,d\mu_tdt, \qquad \int_0^T\int w^2\,d\mu_tdt\le CM.   \]
In particular,
\[
\begin{split}
    |D_{x,t}\chi_{\mathcal E}|
    &\le (1+|w|)\,d\mu,\\
    |D_{x,t}\chi_{\mathcal E}|
       (\mathbb R^{n+1}\times(0,T))
    &\le TM+\sqrt{TM}\sqrt{2M}
     =M\bigl(T+\sqrt{2T}\bigr).
\end{split}
\]
This proves the required absolute continuity and space-time BV regularity,  whereas \eqref{eq:12Holder} gives $\chi_{E_t}\in C^{1/2}_{\mathrm{loc}}([0,\infty);L^1)$. All assumptions of Theorem \ref{thm:L2-to-BV} are now satisfied. Applying with vector velocity $\mathbf H$ completes the proof.
\end{proof}

\begin{remark}
Notice that, as a consequence of Theorem \ref{thm:main}, we have the explicit representation
\begin{equation} \label{eq:time-derivative}
    D_t\chi_{\mathcal E} = - \mathbf H \cdot \frac{d (D\chi_{E_t})}{d\|V_t\|} \, d\|V_t\|dt \qquad \mbox{in the sense of space-time measures}\,,
\end{equation}
that is the limit $w$ in the proof of Theorem \ref{thm:main} admits the explicit characterization $w = -\mathbf{H} \cdot \frac{d (D\chi_{E_t})}{d\|V_t\|}$.
\end{remark}

Under suitable conditions on the initial datum, it is possible to guarantee absence of discrepancy between the Brakke flow and the underlying flat flow. This can be done exploiting the monotonicity of the entropy functional,  after Colding-Minicozzi \cite{ColMin12}. Recall that, for a finite $n$-dimensional measure $\mu$ the entropy is defined by
\[
    \lambda(\mu):=
    \sup_{y\in\mathbb R^{n+1},\,r>0}
    (4\pi r^2)^{-n/2}
    \int e^{-\frac{|x-y|^2}{4r^2}}\,d\mu(x).
\]

\begin{corollary}[Initial entropy below two]\label{cor:entropy-two}
Under the assumptions of Theorem~\ref{thm:main}, suppose additionally that
$\lambda(|D\chi_{E_0}|)<2$.
Let $\{E_t\}_{t\ge 0}$ be a flat flow obtained as $L^1$-limit point of $\{E^{h_j}_t\}_{t\ge 0}$.
Then,
\[
    V_t=V_{E_t}\qquad\text{for almost every }t>0.
\]
In particular,  $E_t$ is a  BV solution, and   for every finite $T$,
\begin{equation} \label{en-conv}
    \lim_{j \to \infty}\int_0^T P(E_t^{h_j})\,dt
    =\int_0^T P(E_t)\,dt.
\end{equation}
Moreover, $\{V_t\}$ is a unit density Brakke flow, and therefore for a.e. $t > 0$ there exists $B_t$ with $\mathcal H^n(B_t)=0$ such that for every $x \in \mathrm{spt}\,\|V_t\| \setminus B_t$ $(x,t)$ is a regular point of the flow.  
\end{corollary}
\begin{proof}
Write $\mu_t:=\|V_t\|$. Huisken monotonicity formula  for the Brakke flow gives
$\lambda(\mu_t)\le\lambda(\mu_0)$. Consider a time $t$ in the full-measure set of times where $V_t$ is integral, and a point $x$ where the approximate tangent plane to $\mu_t$ and the density $\Theta^n(\mu_t,x)$ exist. Notice that $\mu_t$-a.e. $x$ can be chosen as above. Then, at such points it holds
\[\Theta^n(\mu_t,x)\le\lambda(\mu_t)<2\,,\]
as one can immediately see by using the definition of approximate tangent and the fact that the $n$-dimensional Gaussian integrates to $1$ with respect to $\mathcal H^n$ on an $n$-dimensional plane. Since $\Theta^n(\mu_t,x)$ is a positive integer, it must then be equal to one.
The mod$(2)$ compatibility \eqref{mod2-equiv} and domination therefore imply
$V_t=V_{E_t}$. The generalized BV identity  \eqref{eq:BV-gener} and the identification of the mean curvature $\mathbf H_t= \mathbf H_{E_t}$  then  give the
ordinary BV formulation. Finally, \eqref{eq:conv-weight1}  implies that 
$P(E_t^{h_j})\to\mu_t(\mathbb R^{n+1})$; the latter equals $P(E_t)$
for almost every $t$, and thus \eqref{en-conv} follows from dominated convergence. The last assertion follows from \cite[Theorem 3.2]{KasTon}.
\end{proof}

We conclude recording an estimate which is an immediate corollary of Theorem \ref{thm:main}; see \cite[Theorem 2.14]{StuTon}.

\begin{corollary}\label{cor:extinction}
    Under the same assumptions of Theorem \ref{thm:main}, setting $a(t):= \|V_t\|(\R^{n+1})$, it holds
    \begin{equation} \label{eq:holder-sharp}
    |E_t\triangle E_s|
    \le\left[\frac{t-s}{2}\bigl(a(s)^2-a(t)^2\bigr)\right]^{1/2}
    \qquad(0\le s<t).
\end{equation}
    In particular, if the extinction time $T_\star$ is defined by
    \[
    T_\star := \sup\{t \geq 0 \, \colon \, \|V_t\| \neq 0\}\,.
    \]
    then, if $|E_0|>0$ it holds
    \begin{equation} \label{eq:extinction}
        T_\star \geq 2 \left ( \frac{|E_0|}{P(E_0)} \right)^2\,.
    \end{equation}
\end{corollary}

\begin{proof}
Test \eqref{eq:BV-gener} with $\phi(x,\tau) = \eta(\tau)\psi(x)$, where $\eta \in C^1_c([0,\infty))$ with $\eta(\tau) \equiv 1$ on $[s,t]$ and $\psi \in C^1_c(\R^{n+1})$ with $\|\psi\|_\infty \leq 1$. Then by Cauchy-Schwarz  
\[
\left| \int_{\R^{n+1}} (\chi_{E_t}-\chi_{E_s})\,\psi \,dx\right|\leq \int_s^t \sqrt{a(\tau) \,e(\tau)} \, d\tau\,,
\]
where $e(t) = \int |\mathbf H_t|^2\, d\mu_t$. Taking the supremum over $\psi$ on the left then yields
\begin{equation}
    |E_t \triangle E_s| \leq \int_s^t \sqrt{a(\tau) \,e(\tau)} \, d\tau\,.
\end{equation}
The nonincreasing function $a$ satisfies $e(t)\le-a'(t)$ almost everywhere thanks to \eqref{eq:def-Brakke}. Thus we may continue
\[
 \int_s^t \sqrt{a(\tau) \,e(\tau)} \, d\tau \leq \int_s^t \sqrt{-\frac{(a^2(\tau))'}{2}}\,.
\]
We deduce \eqref{eq:holder-sharp} from Cauchy-Schwarz.

For the second part of the statement, observe first that \eqref{eq:holder-sharp} gives
\[
|E_t| \geq |E_0| - \sqrt{t/2} \,P(E_0)\,,
\]
and thus $|E_t| > 0$ for all 
\[
t < 2 \left( \frac{|E_0|}{P(E_0)}\right)^2\,.
\]
For all such $t$, $P(E_t)>0$ by the isoperimetric inequality, and thus $\|V_t\|(\R^{n+1})\geq P(E_t)>0$ by domination. This implies \eqref{eq:extinction} assuming $|E_0|>0$.
\end{proof}

\noindent\textbf{Acknowledgements.}
The authors are supported by the Italian Ministry of University
and Research (MUR) through the FIS 2 project \textbf{SiGmA}: “\textit{Singularities in Geometric
Analysis: Minimal Surfaces and Mean Curvature Flows}”, project code FIS-2023-02962
(CUP G53C25000120001). 

\smallskip

\noindent\textbf{AI usage statement.}
LLM models were used for a preliminary discussion on the problem. In particular, the idea of using the calibration $z$ of Lemma \ref{lem:allmin} has been conceived in an exchange with ChatGPT 5.6. The paper, including all its statements and proofs, has been written solely by the authors. ChatGPT 6 has been used in support of the final proofreading.

\printbibliography

@article{GigaYamauchi93,
  author = {Giga, Yoshikazu and Yama-uchi, Kazuyuki},
  title = {On a lower bound for the extinction time of surfaces moved by mean curvature},
  journal = {Calculus of Variations and Partial Differential Equations},
  year = {1993},
  volume = {1},
  number = {4},
  pages = {417--428},
  doi = {10.1007/BF01206960}
}

@article {KasTon,
    AUTHOR = {Kasai, Kota and Tonegawa, Yoshihiro},
     TITLE = {A general regularity theory for weak mean curvature flow},
   JOURNAL = {Calc. Var. Partial Differential Equations},
  FJOURNAL = {Calculus of Variations and Partial Differential Equations},
    VOLUME = {50},
      YEAR = {2014},
    NUMBER = {1-2},
     PAGES = {1--68},
      ISSN = {0944-2669,1432-0835},
   MRCLASS = {53C44 (49Q20)},
  MRNUMBER = {3194675},
MRREVIEWER = {Chun-Chi\ Lin},
       DOI = {10.1007/s00526-013-0626-4},
       URL = {https://doi.org/10.1007/s00526-013-0626-4},
}

@article {STend,
    AUTHOR = {Stuvard, Salvatore and Tonegawa, Yoshihiro},
     TITLE = {End-time regularity theorem for {B}rakke flows},
   JOURNAL = {Math. Ann.},
  FJOURNAL = {Mathematische Annalen},
    VOLUME = {390},
      YEAR = {2024},
    NUMBER = {3},
     PAGES = {3317--3353},
      ISSN = {0025-5831,1432-1807},
   MRCLASS = {53E10 (49Q20)},
  MRNUMBER = {4803454},
MRREVIEWER = {Jingwen\ Chen},
       DOI = {10.1007/s00208-024-02826-8},
       URL = {https://doi.org/10.1007/s00208-024-02826-8},
}

@article{cdcy,
      title={Sharp dynamical isoperimetric principle for mean curvature flow}, 
      author={Beomjun Choi and Wenkui Du and Seung Chul Park and Junseo Youn},
      year={2026},
      eprint={2610.10461},
      archivePrefix={arXiv},
      primaryClass={math.DG},
      url={https://arxiv.org/abs/2610.10461}, 
      Note={Preprint arXiv:2610.10461},
}

@article {Tas24,
    AUTHOR = {Tashiro, Kiichi},
     TITLE = {Existence of {BV} flow via elliptic regularization},
   JOURNAL = {Hiroshima Math. J.},
  FJOURNAL = {Hiroshima Mathematical Journal},
    VOLUME = {54},
      YEAR = {2024},
    NUMBER = {2},
     PAGES = {233--259},
      ISSN = {0018-2079,2758-9641},
   MRCLASS = {53E10 (28A75)},
  MRNUMBER = {4775228},
       DOI = {10.32917/h2023007},
       URL = {https://doi.org/10.32917/h2023007},
}

@article{Anz,
  title={Pairings between measures and bounded functions and compensated compactness},
  author={Anzellotti, Gabriele},
  journal={Annali di Matematica pura ed applicata},
  volume={135},
  number={1},
  pages={293--318},
  year={1983},
  publisher={Springer}
}

@article{ATW,
  author  = {Almgren, F. and Taylor, J. E. and Wang, L.},
  title   = {Curvature-driven flows: a variational approach},
  journal = {SIAM J. Control Optim.},
  volume  = {31},
  year    = {1993},
  pages   = {387--438},
  doi     = {10.1137/0331020}
}

@article{LS,
  author  = {Luckhaus, S. and Sturzenhecker, T.},
  title   = {Implicit time discretization for the mean curvature flow equation},
  journal = {Calc. Var. PDE},
  volume  = {3},
  year    = {1995},
  pages   = {253--271}
}

@article{ChaCascreNovPoc-intro,
  title={An introduction to total variation for image analysis},
  author={Chambolle, Antonin and Caselles, Vicent and Cremers, Daniel and Novaga, Matteo and Pock, Thomas and others},
  journal={Theoretical foundations and numerical methods for sparse recovery},
  volume={9},
  number={263-340},
  pages={227},
  year={2010}
}

@article{MSS,
  author  = {Mugnai, L. and Seis, C. and Spadaro, E.},
  title   = {Global solutions to the volume-preserving mean-curvature flow},
  journal = {Calc. Var. PDE},
  volume  = {55},
  year    = {2016},
  pages   = {18},
  eprint  = {1502.07232},
  archivePrefix = {arXiv},
  note    = {Numbering refers to the preprint}
}

@book{Maggi,
  author    = {Maggi, F.},
  title     = {Sets of Finite Perimeter and Geometric Variational Problems},
  publisher = {Cambridge University Press},
  year      = {2012}
}

@article{Chambolle,
  author  = {Chambolle, A.},
  title   = {An algorithm for mean curvature motion},
  journal = {Interfaces Free Bound.},
  volume  = {6},
  year    = {2004},
  pages   = {195--218},
  doi     = {10.4171/IFB/97}
}

@article{StuTon,
  title={On the existence of canonical multi-phase Brakke flows},
  author={Stuvard, Salvatore and Tonegawa, Yoshihiro},
  journal={Advances in Calculus of Variations},
  volume={17},
  number={1},
  pages={33--78},
  year={2024},
  publisher={De Gruyter}
}

@article{Tas26,
  title={A note on generalized BV mean curvature flow with a critical forcing term},
  author={Tashiro, Kiichi},
  journal={arXiv preprint arXiv:2608.05760},
  year={2026}
}

@article{KimTon,
  author  = {Kim, Lami and Tonegawa, Yoshihiro},
  title   = {On the mean curvature flow of grain boundaries},
  journal = {Annales de l'Institut Fourier},
  volume  = {67},
  number  = {1},
  pages   = {43--142},
  year    = {2017},
  doi     = {10.5802/aif.3077}
}

@article{MugRog,
  author       = {Mugnai, Luca and R{\"o}ger, Matthias},
  title        = {The {Allen--Cahn} Action Functional in Higher Dimensions},
  journaltitle = {Interfaces and Free Boundaries},
  volume       = {10},
  number       = {1},
  pages        = {45--78},
  year         = {2008},
  doi          = {10.4171/IFB/179}
}

@article{Tas25,
  author       = {Tashiro, Kiichi},
  title        = {Time-Global Existence of Generalized {BV} Flow
                  via the {Allen--Cahn} Equation},
  journaltitle = {Interfaces and Free Boundaries},
  volume       = {27},
  number       = {1},
  pages        = {123--140},
  year         = {2025},
  doi          = {10.4171/IFB/521}
}

@book{Bra,
  author       = {Brakke, Kenneth A.},
  title        = {The Motion of a Surface by Its Mean Curvature},
  series       = {Mathematical Notes},
  number       = {20},
  publisher    = {Princeton University Press},
  location     = {Princeton, NJ},
  year         = {1978}
}

@book{Ton,
  author       = {Tonegawa, Yoshihiro},
  title        = {Brakke's Mean Curvature Flow: An Introduction},
  series       = {SpringerBriefs in Mathematics},
  publisher    = {Springer},
  location     = {Singapore},
  year         = {2019},
  doi          = {10.1007/978-981-13-7075-5}
}

@article{DeG,
  title={New problems on minimizing movements},
  author={De Giorgi, Ennio},
  journal={Boundary value problems for PDE and applications},
  year={1993},
  publisher={Masson}
}

@article{LauUll24,
  title={Generic Level Sets in Mean Curvature Flow are BV Solutions: A. Ullrich, T. Laux},
  author={Ullrich, Anton and Laux, Tim},
  journal={The Journal of Geometric Analysis},
  volume={34},
  number={12},
  pages={375},
  year={2024},
  publisher={Springer}
}

@article{ColMin12,
  title={Generic mean curvature flow I; generic singularities},
  author={Colding, Tobias H and Minicozzi, William P},
  journal={Annals of mathematics},
  pages={755--833},
  year={2012},
  publisher={JSTOR}
}

@article{Whi09,
  author  = {White, Brian},
  title   = {Currents and Flat Chains Associated to Varifolds,
             with an Application to Mean Curvature Flow},
  journal = {Duke Mathematical Journal},
  volume  = {148},
  number  = {1},
  pages   = {41--62},
  year    = {2009},
  doi     = {10.1215/00127094-2009-019},
  mrnumber = {2515099}
}

@article{Whi00,
  title={The size of the singular set in mean curvature flow of mean-convex sets},
  author={White, Brian},
  journal={Journal of the American Mathematical Society},
  volume={13},
  number={3},
  pages={665--695},
  year={2000}
}

@article{LauOtt,
  author  = {Laux, Tim and Otto, Felix},
  title   = {Brakke's Inequality for the Thresholding Scheme},
  journal = {Calculus of Variations and Partial Differential Equations},
  volume  = {59},
  year    = {2020},
  number  = {1},
  pages   = {39},
  doi     = {10.1007/s00526-019-1673-5}
}

@article{HenLau,
  title={A new varifold solution concept for mean curvature flow: Convergence of the Allen--Cahn equation and weak-strong uniqueness},
  author={Hensel, Sebastian and Laux, Tim},
  journal={Journal of Differential Geometry},
  volume={130},
  number={1},
  pages={209--268},
  year={2025},
  publisher={Lehigh University}
}

@article{LauPoi,
  title={De Giorgi varifold solutions to Mean Curvature Flow: a minimizing movements approach},
  author={Laux, Tim and Poiatti, Andrea},
  journal={arXiv preprint arXiv:2607.03930},
  year={2026}
}

@article{JulNii24,
  title={Consistency of the flat flow solution to the volume preserving mean curvature flow},
  author={Julin, Vesa and Niinikoski, Joonas},
  journal={Archive for Rational Mechanics and Analysis},
  volume={248},
  number={1},
  pages={1},
  year={2024},
  publisher={Springer}
}

@article{ChaDegMor24,
  title={Minimizing movements for anisotropic and inhomogeneous mean curvature flows},
  author={Chambolle, Antonin and De Gennaro, Daniele and Morini, Massimiliano},
  journal={Advances in Calculus of Variations},
  volume={17},
  number={4},
  pages={1095--1129},
  year={2024},
  publisher={De Gruyter}
}

@book{AmbGigSav,
  author    = {Ambrosio, Luigi and Gigli, Nicola and Savar{\'e}, Giuseppe},
  title     = {Gradient Flows: In Metric Spaces and in the Space of Probability Measures},
  series    = {Lectures in Mathematics ETH Z{\"u}rich},
  edition   = {2},
  publisher = {Birkh{\"a}user},
  address   = {Basel},
  year      = {2008},
  doi       = {10.1007/978-3-7643-8722-8},
  isbn      = {978-3-7643-8721-1}
}

@article{Allard,
  author  = {Allard, W. K.},
  title   = {On the first variation of a varifold},
  journal = {Ann. of Math.},
  volume  = {95},
  year    = {1972},
  pages   = {417--491}
}

@article{CraDec,
  title={Anzellotti's pairing theory and the Gauss--Green theorem},
  author={Crasta, Graziano and De Cicco, Virginia},
  journal={Advances in Mathematics},
  volume={343},
  pages={935--970},
  year={2019},
  publisher={Elsevier}
}

@article{Bel,
  title={Extensions of schoen--simon--yau and schoen--simon theorems via iteration {\'a} la de giorgi},
  author={Bellettini, Costante},
  journal={Inventiones mathematicae},
  volume={240},
  number={1},
  pages={1--34},
  year={2025},
  publisher={Springer}
}

@article{Ilmanen94,
  author = {Ilmanen, Tom},
  title = {Elliptic regularization and partial regularity for motion by mean curvature},
  journal = {Mem. Amer. Math. Soc.},
  fjournal = {Memoirs of the American Mathematical Society},
  volume = {108},
  year = {1994},
  number = {520},
  pages = {x+90},
  issn = {0065-9266},
  coden = {MAMCAU},
  mrclass = {49Q20 (53A10 58E12)},
  mrnumber = {1196160 (95d:49060)},
  mrreviewer = {Wei Yue Ding},
  doi = {10.1090/memo/0520},
  url = {http://dx.doi.org/10.1090/memo/0520},
}

@article{Jul24,
  title = {Flat flow solution to the mean curvature flow with volume constraint},
  author = {Julin, Vesa},
  journal = {Advances in Calculus of Variations},
  volume = {17},
  number = {4},
  pages = {1543--1555},
  year = {2024},
  publisher = {De Gruyter},
}

@article{EvaSpr,
  author = {Evans, Lawrence C. and Spruck, Joel},
  title = {Motion of level sets by mean curvature. {I}},
  journal = {J. Differential Geom.},
  fjournal = {Journal of Differential Geometry},
  volume = {33},
  year = {1991},
  number = {3},
  pages = {635--681},
  issn = {0022-040X},
  coden = {JDGEAS},
  mrclass = {35K55 (35B05 35D05 53A10)},
  mrnumber = {92h:35097},
  mrreviewer = {Friedrich Sauvigny},
}

@article{CheGigGot,
  author = {Chen, Y. G. and Giga, Y. and Goto, S.},
  title = {Uniqueness and existence of viscosity solutions of generalized mean curvature flow equations},
  journal = {J. Differential Geom.},
  fjournal = {Journal of Differential Geometry},
  volume = {33},
  year = {1991},
  number = {3},
  pages = {749--786},
}

@article{DePLau,
  author = {De Philippis, G. and Laux, T.},
  title = {Implicit time discretization for the mean curvature flow of mean convex sets},
  journal = {Ann. Sc. Norm. Super. Pisa Cl. Sci. (5)},
  fjournal = {Annali della Scuola Normale Superiore di Pisa. Classe di Scienze. Serie V},
  volume = {21},
  year = {2020},
  pages = {911--930},
}

@article{ChaMorPonCPAM17,
  author = {Chambolle, Antonin and Morini, Massimiliano and Ponsiglione, Marcello},
  title = {Existence and uniqueness for a crystalline mean curvature flow},
  journal = {Comm. Pure Appl. Math.},
  fjournal = {Communications on Pure and Applied Mathematics},
  volume = {70},
  year = {2017},
  number = {6},
  pages = {1084--1114},
  issn = {0010-3640},
  mrclass = {35K55 (35A01 35A02 35K93 53C44)},
  mrnumber = {3639320},
  mrreviewer = {Xiaoliu Wang},
  doi = {10.1002/cpa.21668},
  url = {https://doi.org/10.1002/cpa.21668},
}

@article{EvansSpruckIV,
  author = {Evans, L. C. and Spruck, J.},
  title = {Motion of level sets by mean curvature. {IV}},
  journal = {J. Geom. Anal.},
  fjournal = {The Journal of Geometric Analysis},
  volume = {5},
  year = {1995},
  number = {1},
  pages = {77--114},
  issn = {1050-6926},
  mrclass = {35K55 (49Q15 53A07 58E99)},
  mrnumber = {96a:35077},
  mrreviewer = {Ben Andrews},
}

@article{RuOsFa:92,
  author = {Rudin, L. I. and Osher, S. and Fatemi, E.},
  title = {Nonlinear total variation based noise removal algorithms},
  journal = {Physica D},
  year = {1992},
  optkey = {},
  volume = {60},
  optnumber = {},
  pages = {259--268},
  optmonth = {},
  optnote = {},
  optannote = {},
}

@book{GilbargTrudinger,
  author = {Gilbarg, D. and Trudinger, N. S.},
  isbn = {3-540-41160-7},
  mrclass = {35-02 (35Jxx)},
  mrnumber = {1814364},
  note = {Reprint of the 1998 edition},
  pages = {xiv+517},
  publisher = {Springer-Verlag, Berlin},
  series = {Classics in Mathematics},
  title = {Elliptic partial differential equations of second order},
  year = {2001},
}

@article{BelKho,
  title = {Minimizing movements for mean curvature flow of droplets with prescribed contact angle},
  author = {Bellettini, Giovanni and Kholmatov, Sh Yu},
  journal = {Journal de Math{\'e}matiques Pures et Appliqu{\'e}es},
  volume = {117},
  pages = {1--58},
  year = {2018},
  publisher = {Elsevier},
}

@article{Soner93,
  author = {Soner, Halil Mete},
  title = {Motion of a set by the curvature of its boundary},
  journal = {J. Differential Equations},
  fjournal = {Journal of Differential Equations},
  volume = {101},
  year = {1993},
  number = {2},
  pages = {313--372},
  issn = {0022-0396},
  coden = {JDEQAK},
  mrclass = {58E50 (35B99 53A10 58E12 73B30)},
  mrnumber = {1204331 (94b:58038)},
  mrreviewer = {Anders Linn{\'e}r},
  doi = {10.1006/jdeq.1993.1015},
  url = {http://dx.doi.org/10.1006/jdeq.1993.1015},
}

\end{document}